\documentclass[12pt]{amsart}
\usepackage{fullpage,url,amssymb,enumerate,colonequals}
\usepackage[all,pdf]{xy} % for complicated commutative diagrams, use pdf driver

\usepackage{comment}
\usepackage{color}
\usepackage{charter,euler} % for better looks

\DeclareSymbolFont{bchoperators}{T1}{bch}{m}{n}
\SetSymbolFont{bchoperators}{bold}{T1}{bch}{b}{n}
\makeatletter
\renewcommand{\operator@font}{\mathgroup\symbchoperators}
\makeatother

\usepackage{titlesec} % change small caps in section titles to boldface
\titleformat{\section}{\normalfont\bfseries\filcenter}{\thesection}{1em}{}

\usepackage[pdftex]{epsfig}
\newcommand{\Gr}[2]{\psfig{file=#1.pdf,width=#2}}

\newcommand{\C}{\mathbb{C}}
\newcommand{\Q}{\mathbb{Q}}
\newcommand{\Z}{\mathbb{Z}}

\newcommand{\au}{\underline{a}}
\newcommand{\ao}{\overline{a\text{\tiny\strut}}}
\newcommand{\bu}{\underline{b}}
\newcommand{\bo}{\overline{b}}

\newcommand{\eps}{\varepsilon}

\numberwithin{equation}{section}
\newtheorem{theorem}{Theorem}[section]
\newtheorem{lemma}[theorem]{Lemma}
\newtheorem{corollary}[theorem]{Corollary}
\newtheorem{proposition}[theorem]{Proposition}

\theoremstyle{definition}
\newtheorem{definition}[theorem]{Definition}
\newtheorem{question}[theorem]{Question}
\newtheorem{conjecture}[theorem]{Conjecture}
\newtheorem{example}[theorem]{Example}

\theoremstyle{remark}
\newtheorem{remark}[theorem]{Remark}

\definecolor{darkgreen}{rgb}{0,0.5,0}
\usepackage[
        colorlinks, citecolor=darkgreen,
        backref,
        pdfauthor={William Sawin, Mark Shusterman, Michael Stoll},
        pdftitle={Irreducibility of polynomials with a large gap},
]{hyperref}
\usepackage[alphabetic,backrefs,lite]{amsrefs} % for bibliography

\begin{document}

\title{Irreducibility of polynomials with a large gap}

\author{William Sawin}
\address{ETH Inst.~f\"ur Theoretische Studien,
         Clausiusstrasse 47,
         8092 Z\"urich,
         Switzerland}
\email{william.sawin@math.ethz.ch}
\urladdr{http://williamsawin.com/}

\author{Mark Shusterman}
\address{Raymond and Beverly Sackler School of Mathematical Sciences,
         Tel-Aviv University,
         Tel-Aviv,
         Israel}
\email{markshus@mail.tau.ac.il}

\author{Michael Stoll}
\address{Mathematisches Institut,
         Universit\"at Bayreuth,
         95440 Bayreuth, Germany.}
\email{Michael.Stoll@uni-bayreuth.de}
\urladdr{http://www.mathe2.uni-bayreuth.de/stoll/}

\date{\today}

\begin{abstract}
  We generalize an approach  from a 1960~paper by~Ljunggren, leading to
  a practical algorithm that determines the set of $N > \deg c + \deg d$
  such that the polynomial
  \[ f_N(x) = x^N c(x^{-1}) + d(x) \]
  is irreducible over~$\Q$, where $c, d \in \Z[x]$ are polynomials
  with nonzero constant terms and satisfying suitable conditions.
  As an application, we show that $x^N - k x^2 + 1$ is irreducible
  for all $N \ge 5$ and $k \in \{3, 4, \ldots, 24\} \setminus \{9, 16\}$.
  We also give a complete description of the factorization of
  polynomials of the form $x^N + k x^{N-1} \pm (l x + 1)$ with $k, l \in \Z$,
  $k \neq l$.
\end{abstract}

\subjclass[2010]{11R09, 12E05, 11C08, 13P05}

\maketitle

%==========================================================================

\section{Introduction} \label{S:intro}

Providing irreducibility criteria for integral polynomials is by now a classical
topic, as can be seen for instance from the books \cite{Prasolov} by~Prasolov
or \cite{SchinzelBook} by~Schinzel.
Yet, the irreducibility of most polynomials cannot be established using the
classical techniques, and many problems remain open. One example is the irreducibility
of random polynomials, as studied for instance in~\cite{BarKoz}. Another challenge,
motivated by the calculation of Galois groups, lies in finding irreducibility criteria
for trinomials. Indeed, in various works such as~\cites{CoMoSa97, CoMoSa99, MoSa96, Os87},
Galois groups are calculated under an irreducibility assumption.
The purpose of this work is to obtain such irreducibility criteria.
More generally, we consider polynomials ``with a large gap'', by which we
mean polynomials of the form
\[ f_N(x) = x^N c(x^{-1}) + d(x) \,, \]
where $c$ and $d$ are fixed polynomials in~$\Z[x]$ with $c(0), d(0) \neq 0$
and we are interested in the irreducibility of~$f_N$ for large~$N$.
Polynomials of this type and their factorization into irreducibles have
been considered in various contexts;
see for example~\cites{Schinzel1967,FFK2000,FM2004,DFV2013,HVW2013}.
The main contributions of this paper are to give an improved bound for~$N$ such
that the factorization of~$f_N$ can be controlled and to present an
algorithm that can in many cases determine the factorizations of all~$f_N$.
This requires $c$ and~$d$ to satisfy some additional conditions.

\medskip

For a polynomial~$f$, we set $\tilde{f}(x) = x^{\deg f} f(x^{-1})$
and we say that $f$ is \emph{reciprocal} if $\tilde{f} = \pm f$.
If $f \neq 0$, we define the \emph{non-reciprocal part} of~$f \in \Z[x]$ to be
$f$ divided by all reciprocal and non-constant irreducible factors in its
prime factorization over~$\Z[x]$. Similarly, we define the \emph{non-cyclotomic part}
of~$f$ to be $f$ divided by all irreducible factors that are cyclotomic polynomials.
(Both of these are only defined up to a sign, but the sign is irrelevant
for our purposes.)
Since we are interested in the irreducibility of~$f_N$ above, we can
always assume that $\gcd_{\Z[x]}(\tilde{c}, d) = 1$, since otherwise
this gcd will give a trivial divisor of~$f_N$ for all~$N$. We will in addition
assume that $f_N$ is not reciprocal, which is equivalent to $c \neq \pm d$.

Note that ``irreducible'' in this paper always means ``irreducible over~$\Q$''.

There are systematically occurring nontrivial factorizations of~$f_N$.

\begin{definition}
  A pair $(c,d)$ of polynomials $c, d \in \Z[x]$ with $c(0), d(0) \neq 0$
  is \emph{Capellian} when $-d(x)/c(x^{-1})$ is a $p$th power in~$\Q(x)$
  for some prime~$p$ or $d(x)/c(x^{-1})$ is $4$~times a fourth power in~$\Q(x)$.
\end{definition}

The name honors Alfredo Capelli, who showed that for $a$ in
a field~$K$, the polynomials $x^N - a$ are irreducible in~$K[x]$
for all~$N$ if and only if $a$ is not a $p$th power for some prime~$p$
or $-4$~times a fourth power in~$K$; see~\cite{Capelli}.
So when $(c,d)$ is Capellian (and only then), we get factorizations of~$f_N$
coming from factorizations of $y^n + d(x)/c(x^{-1})$. We will restrict
to non-Capellian pairs $(c,d)$ in the following, but we note that the
results below continue to hold when $(c,d)$ is Capellian and $N$ is not
a multiple of~$p$ (when $-d(x)/c(x^{-1})$ is a $p$th power) or~$4$
(when $d(x)/c(x^{-1})$ is $4$~times a fourth power).

The main general result is a consequence of work by Schinzel.

\begin{theorem}[Schinzel] \label{T:Schinzel}
  Let $c, d \in \Z[x]$ with $c(0), d(0) \neq 0$.
  Assume that $\gcd_{\Z[x]}(\tilde{c}, d) = 1$, that $c \neq \pm d$ and
  that $(c,d)$ is not Capellian. Then there is a bound $N_0$ depending
  only on $c$ and~$d$ such that
  for $N > N_0$, the non-reciprocal part of~$f_N$ is irreducible.
\end{theorem}

\begin{proof}
  This can be deduced from Theorem~74 in~\cite{SchinzelBook};
  see also~\cite{Schinzel1969}*{Theorem~2}, where the result is stated
  over~$\Q$ and explicit bounds are given. Let
  \[ F(x_1, x_2) = x_2 \tilde{c}(x_1) + d(x_1)\,;
     \quad\text{then}\quad F(x,x^{N-\deg c}) = f_N(x) \,.
  \]
  Since $(c,d)$ is non-Capellian, we deduce that
  $F(y_1^{m_{11}} y_2^{m_{21}}, y_1^{m_{12}} y_2^{m_{22}})$
  is irreducible for any matrix $M = (m_{ij}) \in \Z^{2 \times 2}$ of rank~$2$
  and such that $(1, n)$ is an integral linear combination of the
  rows of~$M$, for some~$n$. This is a fairly easy consequence
  of Capelli's Theorem~\cite{Capelli}.

  Schinzel's theorem tells us that for some $1 \le r \le 2$, there is
  a matrix~$M = (m_{ij}) \in \Z^{r \times 2}$ of rank~$r$ such that $(1, N-\deg c)$
  is an integral linear combination of the rows of~$M$ and such that
  the entries of~$M$ are bounded by a constant~$C$ only depending on~$F$.
  When $r = 1$, then up to a sign, $M = (1 \; N-\deg c)$, so $N \le \deg c + C$.
  When $r = 2$, then
  \[ \tilde{F}(y_1, y_2) = F(y_1^{m_{11}} y_2^{m_{21}}, y_1^{m_{12}} y_2^{m_{22}}) \]
  is irreducible by the above. Our assumptions on $c$ and~$d$ imply that
  $f_N$ is not reciprocal, which implies that $L \tilde{F} = \tilde{F}$
  in the notation of~\cite{Schinzel1969}. Then Schinzel's theorem says
  that the factorization of the non-reciprocal part $L f_N$ of~$f_N$
  corresponds to the factorization of $L \tilde{F} = \tilde{F}$. But the
  latter is irreducible, hence the non-reciprocal part of~$f_N$ is irreducible
  as well. So the claim holds with $N_0 = \deg c + C$, and $C$ depends only on~$F$,
  which in turn depends only on $c$ and~$d$.
\end{proof}

For a polynomial $f \in \Z[x]$, we define its \emph{weight}~$\|f\|$
to be the squared Euclidean length of its coefficient vector (i.e., the
sum of the squares of the coefficients).
The explicit bounds given in~\cite{Schinzel1969} then amount to
\[ N_0 \le \deg c + \exp\bigl(\tfrac{5}{16} \cdot 2^{(\|c\|+\|d\|)^2}\bigr)
                      (2 + \max\{2, (\deg c)^2, (\deg d)^2\})^{\|c\|+\|d\|} \,.
\]

Now consider a reciprocal irreducible factor~$h$ of~$f_N$.
Then $h = \pm \tilde{h}$ also divides~$\tilde{f}_N$, so $h$ divides
$\gcd(f_N, \tilde{f}_N)$, which in turn divides
\[ x^{\deg c} \tilde{d}(x) f_N(x) - x^{\deg d} \tilde{c}(x) \tilde{f}_N(x)
     = x^{\deg c} d(x) \tilde{d}(x) - x^{\deg d} c(x) \tilde{c}(x)
     = x^n r(x) \,,
\]
where $n \in \Z_{\ge 0}$ and~$r \in \Z[x]$ are such that $r(0) \neq 0$.
Since $f_N(0) \neq 0$, it follows that $h$ divides~$r$, and the assumptions
$\gcd_{\Z[x]}(\tilde{c}, d) = 1$ and $c \neq \pm d$ guarantee that $r \neq 0$.
So any reciprocal irreducible factor~$h$ of~$f_N$ divides the fixed polynomial~$r$
of degree at  most $2m$, where $m = \max\{\deg c, \deg d\}$. By Lemma~\ref{L:divbound} below,
it follows that $h$ must be a cyclotomic polynomial when
\[ N > N_1 = \deg c + \deg d
              + \begin{cases}
                  \dfrac{2 m}{\log \theta} \log(\|c\| + \|d\|) & \text{if $m \le 27$,} \\
                   m (\log 6 m)^3 \log(\|c\| + \|d\|) & \text{otherwise,}
                \end{cases}
\]
where Lehmer's constant $\theta \approx 1.17628$ is defined in Section~\ref{S:lower}.
This leads to the following.

\begin{corollary} \label{C:noncyc}
  Under the assumptions of Theorem~\ref{T:Schinzel}, if $N > \max\{N_0, N_1\}$,
  then the non-cyclotomic part of~$f_N$ is irreducible.
\end{corollary}

By the above, every cyclotomic divisor of~$f_N$ must divide~$r$, which
leads to a finite set of possible cyclotomic divisors. If a cyclotomic
polynomial~$\Phi_n$ divides~$f_N$ for some~$N$, then it clearly divides~$f_{N'}$
if and only if $N' \equiv N \bmod n$. So each cyclotomic polynomial that occurs
as a factor of some~$f_N$ does so exactly for $N$ in some arithmetic progression.
Whence:

\begin{corollary}
  Under the assumptions of Theorem~\ref{T:Schinzel}, the set of $N > \deg c + \deg d$
  such the polynomial~$f_N$ is irreducible is the complement of the union
  of a finite set with a finite union of arithmetic progressions.
  Both the finite set and the finite union of arithmetic progressions can
  be determined effectively.
\end{corollary}

\begin{proof}
  The first statement is clear from the discussion above. It remains to
  prove effectivity. Given $c$ and~$d$, we compute~$r$ and find all
  cyclotomic polynomials~$\Phi_n$ dividing~$r$. For each such~$\Phi_n$,
  we check if $\Phi_n$ divides~$f_N$ for a complete set of representatives
  $N > \deg c + \deg d$ of the residue classes mod~$n$. If it does, then
  it does so for precisely one representative, which gives rise to one
  of the arithmetic progressions. Otherwise there is no arithmetic
  progression coming from~$\Phi_n$. We obtain the finite set by checking
  the irreducibility of~$f_N$ for all $\deg c + \deg d < N \le \max\{N_0,N_1\}$
  by standard algorithms. Note that there is an explicit and hence effective
  bound on~$\max\{N_0,N_1\}$.
\end{proof}

There are examples where the finite union
of residue classes is all of $\Z_{>\deg c + \deg d}$, but without a single
cyclotomic polynomial dividing all~$f_N$ (as is the case for $x^N - 2x + 1$).
The following example
is due to Schinzel~\cite{SchinzelBook}*{Remark~3 in Section~6.4}. Take
\[ c = 12\,, \qquad d = 3 x^9 + 8 x^8 + 6 x^7 + 9 x^6 + 8 x^4 + 3 x^3 + 6 x + 5 \,. \]
We have the cyclotomic factors
\begin{align*}
  1 + x         & \qquad \text{if $N \equiv 1 \bmod 2$,} \\
  1 + x + x^2   & \qquad \text{if $N \equiv 2 \bmod 3$,} \\
  1 + x^2       & \qquad \text{if $N \equiv 2 \bmod 4$,} \\
  1 - x + x^2   & \qquad \text{if $N \equiv 4 \bmod 6$ \quad and} \\
  1 - x^2 + x^4 & \qquad \text{if $N \equiv 0 \bmod 12$;}
\end{align*}
the remaining part of~$f_N$ is irreducible for all $N > 9$ as can be
shown by our algorithm.

Of course, the explicit bound given by Schinzel is much too large to
make this procedure practical even for $c$ and~$d$ of very small degree and weight.
There are results that improve on this bound, the best of which seems to be the
following; see~\cite{FFK2000}.

\begin{theorem}[Filaseta, Ford, Konyagin] \label{T:FFK}
  In Theorem~\ref{T:Schinzel}, we can take
  \[ N_0 \le N_{\text{\upshape FFK}}
         = \deg c + 2 \max\bigl\{5^{4w-15}, \max\{\deg c, \deg d\} (5^{2w-8} + \tfrac{1}{4})\bigr\}\,,
  \]
  where $w = \|c\| + \|d\| + t$ and $t$ is the number of terms in $c$ and~$d$.
\end{theorem}

In~\cite{DFV2013} a similar result is shown for the non-cyclotomic part of~$f_N$,
but their bound~$B_2$ is much larger than our~$N_1$.

Our main contribution is the following.

\begin{theorem} \label{T:main}
  Let $c, d \in \Z[x]$ with $c(0), d(0) \neq 0$.
  We assume that $\gcd_{\Z[x]}(\tilde{c}, d) = 1$, that $c \neq \pm d$ and
  that $(c,d)$ is robust in the sense of Definition~\ref{D:robust}. Then
  we can take
  \[ N_0 \le (1 + \deg c + \deg d) 2^{\|c\| + \|d\|} \]
  in Theorem~\ref{T:Schinzel} and in Corollary~\ref{C:noncyc}.
\end{theorem}

The main disadvantage of our result compared to Theorem~\ref{T:FFK} is the
additional condition that $(c,d)$ is robust. It is satisfied in many cases
of interest (for example, when $c = 1$ and $d$ is irreducible and primitive
in the sense that the gcd of its coefficients is~$1$),
but not always. There are some advantages that make up for this, though:
\begin{enumerate}[1.]\addtolength{\itemsep}{3pt}
  \item Our bound for~$N_0$ is considerably smaller than~$N_{\text{FFK}}$.
  \item We describe an algorithm that computes a suitable~$N_0$ for any
        given robust pair~$(c,d)$; this bound is usually much smaller
        than~$(1 + \deg c + \deg d) 2^{\|c\| + \|d\|}$.
        In Section~\ref{S:m0}, we present some evidence indicating that the worst-case
        growth of the bound obtained by our algorithm
        should be quadratic instead of exponential in~$\|c\| + \|d\|$.
  \item When the degrees and weights of $c$ and~$d$ are reasonably small,
        our algorithm is entirely practical and can be used to produce the complete
        list of irreducible factors of~$f_N$ of degree~$\le N/2$
        (the degree will in fact be uniformly bounded)
        for all $N > \deg c + \deg d$. See Section~\ref{S:ex} for examples.
\end{enumerate}

We remark that for polynomials $c, d$ with coefficients in~$\{0,1\}$,
even stronger results can be shown; see~\cite{FM2004}, where a result
similar to Theorem~\ref{T:Schinzel} is obtained with a bound that is linear
in $\max\{\deg c, \deg d\}$. Even assuming that $(c,d)$ is robust
(which is not always the case), examples show that we cannot hope for
better than quadratic bounds with our method.

Considering the various upper bounds on~$N_0$, we may wonder what the
optimal bound might be. Let us define $N_{\text{opt}}(\delta, w)$
to be the smallest value of~$N_0$ such that the statement of Theorem~\ref{T:Schinzel}
holds for all $c, d$ with $\deg c + \deg d = \delta$ and $\|c\| + \|d\| \le w$.
Note that $N_{\text{opt}}(\delta, w)$ is well-defined, since there are
only finitely many such pairs~$(c,d)$. Since a factorization of~$f_N$
leads to an analogous factorization of~$f_N(x^k)$ for any $k \ge 1$,
we see that $N_{\text{opt}}(k \delta, w) \ge k N_{\text{opt}}(\delta, w)$.
To get a lower bound on~$N_{\text{opt}}(\delta, w)$, we can fix an
irreducible polynomial~$p$ and try to find $c, d$ such that $p$ divides~$f_N$
for large~$N$. For example, defining the Fibonacci numbers as usual by
$F_0 = 0$, $F_1 = 1$, $F_{n+1} = F_n + F_{n-1}$, we see easily that
\[ x^2 - x - 1 \quad\text{divides}\quad x^N - F_N x - F_{N-1} \,, \]
which implies that
\[ N_{\text{opt}}(\delta, w) \ge \frac{\delta \log w}{2 \log \phi} \,, \]
where $\phi = (1+\sqrt{5})/2$ is the golden ratio. As long as
$\deg p \le \delta + 1$, we can always write $\alpha^N$ as $-d(\alpha)$
with $\deg d \le \delta$, where $\alpha$ is a root of~$p$, which gives
that $p$ divides $x^N + d(x)$. By Lemma~\ref{L:divbound}, this will
give a lower bound on~$N_{\text{opt}}(\delta, w)$ that cannot be larger than
\[ \delta + \frac{\delta + 1}{\log \theta} \log w \,, \]
assuming that Lehmer's constant is the optimal lower bound for~$M(p)$.
(Asymptotically, we can replace $\log w$ by~$\frac{1}{2} \log w$;
this comes from using the better estimate
\[ \max\{s(c), s(d)\} \le \sqrt{\max\{\deg c, \deg d\}} \sqrt{\|c\|+\|d\|} \]
in the proof of Lemma~\ref{L:divbound}.)

So if we want to show that $N_{\text{opt}}(\delta, w)$ grows faster
than $\delta \log w$, then we have to work with polynomials~$p$
such that $(\deg p)/(\log M(p))$ grows faster than~$\delta$.
But as soon as the difference between $\deg p$ and~$\delta$ is large,
the coefficients of~$p$ must satisfy many algebraic equations
of a degree that grows linearly with~$N$ in order to have a pair~$(c,d)$
such that $p$ divides~$f_N$. This appears difficult to accomplish
in a systematic way. So we would like to propose the following
questions as a motivation for further study.

\begin{question}
  Fix $\delta > 1$. Is $\dfrac{N_{\text{opt}}(\delta, w)}{\log w}$ bounded?
\end{question}

\begin{question}
  Is perhaps even $\dfrac{N_{\text{opt}}(\delta, w)}{\delta \log w}$ bounded?
\end{question}

One possibly interesting data point is $N_{\text{opt}}(1, 26) \ge 14$, coming
from
\[ x^{14} + 4 x + 3 = (x + 1) (x^3 - x^2 + 1) (x^{10} + x^8 - x^7 - 2x^5 - x^3 + 2x^2 + x + 3) \,. \]

\medskip

The key idea of our approach for proving Theorem~\ref{T:main}
is based on a neat trick due to Ljunggren~\cite{Ljunggren}
(see also~\cite{Prasolov}*{Section~2.3}), which
can be used to show (for example) that $x^n - x - 1$ is irreducible for all~$n$.
After proving the lemma that provides the bound~$N_1$ for Corollary~\ref{C:noncyc}
in Section~\ref{S:lower},
we recall Ljunggren's approach in Section~\ref{S:Lj} and then develop
our generalization in Section~\ref{S:main}.
In Section~\ref{S:m0}, we discuss the growth of a quantity~$m_0$ that depends
on $c$ and~$d$; this quantity enters into the bound~$N_0$ in Theorem~\ref{T:main}.
In Section~\ref{S:algo}, we describe an improvement of the algorithm used
in determining~$m_0$.
We end with a collection of sample applications in Section~\ref{S:ex}.
For example, we answer the following question that was asked
on~\emph{MathOverflow}~\cite{MO} and was a motivation for this work:
\begin{center}
  \emph{Are the polynomials $x^{2k+1} - 7 x^2 + 1$ irreducible over~$\Q$?}
\end{center}
For some families of pairs~$(c,d)$, we can use our algorithm to produce
a uniformly small bound~$N_0$. This leads to
a complete analysis of the factorization patterns of polynomials
of the form $x^N + k x^{N-1} \pm (l x + 1)$, where $k$ and~$l$ are distinct
integers.

\medskip

We would like to thank Michael Filaseta for some useful
comments on an earlier version of this paper and Umberto Zannier
for a very helpful discussion of Schinzel's contributions and relations
with unlikely intersections.

%==========================================================================

\section{An application of heights} \label{S:lower}

In this section, we provide the result necessary to obtain the bound~$N_1$
in Corollary~\ref{C:noncyc}. This is also relevant for the discussion
of lower bounds on~$N_0$.

For a polynomial
\[ f(x) = c \prod_{j=1}^n (x - \alpha_j) \in \C[x] \,, \]
its \emph{Mahler measure} is defined to be
\[ M(f) = |c| \prod_{j=1}^n \max\{1, |\alpha_j|\} \,; \]
see for example~\cite{SchinzelBook}*{Section~3.4}.
If $f \in \Z[x]$, then $M(f) \ge 1$, and it is a known fact that
$M(f) = 1$ if and only if $f$ is (up to a sign) a product of a power of~$x$
and cyclotomic polynomials. It is an open question (``Lehmer's problem'')
whether there is a lower bound $> 1$ for~$M(f)$ when $f$ is not of this form.
The record polynomial in this respect was already found by Lehmer; it is
\[ x^{10} + x^9 - x^7 - x^6 - x^5 - x^4 - x^3 + x + 1 \,, \]
and its Mahler measure is \emph{Lehmer's constant} $\theta \approx 1.17628$
(which is its unique real root $> 1$).
By~\cite{MRW2008}, the smallest Mahler measure of a non-cyclotomic polynomial
of degree up to~$54$ is indeed~$\theta$. In general, the best currently
known explicit bound seems to be due to Voutier~\cite{Voutier1996};
a slightly less good, but simpler variant is the estimate
(Corollary~2 in~\cite{Voutier1996})
\[ \log M(p) > \frac{2}{(\log (3 \deg p))^3} \]
for $p \in \Z[x]$ irreducible with $M(p) > 1$.
If $p$ is non-reciprocal, then we have the stronger (and optimal) bound
\[ M(p) \ge \theta_0 \,, \]
where $\theta_0 \approx 1.3247$
is the unique real root and also the Mahler measure of $x^3 - x - 1$;
this result is due to Smyth~\cite{Smyth1971}
(see also~\cite{SchinzelBook}*{Corollary~5 in Section~6.1}).

The \emph{absolute logarithmic Weil height} of an algebraic number~$\alpha$
can be defined as
\[ h(\alpha) = \frac{\log M(p_\alpha)}{\deg p_\alpha} \,, \]
where $p_\alpha \in \Z[x]$ is the minimal polynomial of~$\alpha$ over~$\Q$
scaled so that its coefficients are coprime integers. We will simply
call $h(\alpha)$ the \emph{height} of~$\alpha$.
There is an alternative definition of~$h(\alpha)$ in terms of a complete
system of absolute values on any finite extension of~$\Q$ containing~$\alpha$,
from which one can easily deduce the following properties.
\begin{enumerate}[(1)]
  \item For $N \in \Z$ and $\alpha \in \bar{\Q}$, we have that $h(\alpha^N) = |N| \, h(\alpha)$.
  \item For $p, q \in \Z[x]$ and $\alpha \in \bar{\Q}$, we have that
        \[ h\Bigl(\frac{p(\alpha)}{q(\alpha)}\Bigr)
             \le \log\max\{s(p),s(q)\} + \max\{\deg p, \deg q\} \, h(\alpha) \,.
        \]
        Here $s(f)$ is the sum of the absolute values of the coefficients of~$f$.
\end{enumerate}
See for example~\cite{HindrySilverman}*{Part~B}.

\begin{lemma} \label{L:divbound}
  Let $p \in \Z[x]$ be irreducible, non-constant and non-cyclotomic. If
  $\gcd(\tilde{c}, d) = 1$ and $p$ divides~$f_N$, then
  \[ N \le \deg c + \deg d + \frac{\deg p}{\log M(p)} \log(\|c\| + \|d\|) \,. \]
  This implies that
  \[ N \le \deg c + \deg d +
          \begin{cases}
            \dfrac{\deg p}{\log \theta_0} \log(\|c\| + \|d\|) & \text{if $p$ is non-reciprocal,} \\[12pt]
            \dfrac{\deg p}{\log \theta} \log(\|c\| + \|d\|) & \text{if $\deg p \le 54$,} \\[12pt]
            \frac{1}{2} (\deg p)(\log (3 \deg p))^3 \log(\|c\| + \|d\|) & \text{otherwise.}
          \end{cases}
  \]
\end{lemma}

\begin{proof}
  We can assume that $p$ is primitive (i.e., the gcd of its coefficients is~$1$).
  Let $\alpha \in \bar{\Q}$ be a root of~$p$. The condition $\gcd(\tilde{c}, d) = 1$
  implies that when $p$ divides~$f_N$, then $p$ does not divide~$\tilde{c}$.
  So if $p$ divides~$f_N$, then $f_N(\alpha) = 0$ and $\alpha \neq 0$, $c(\alpha^{-1}) \neq 0$,
  which implies that $\alpha^N = -d(\alpha)/c(\alpha^{-1})$.
  Since $p$ is not cyclotomic and $\alpha \neq 0$, we have that $h(\alpha) > 0$.
  From the two properties of the height stated above, we deduce that
  \begin{align*}
    N h(\alpha) &= h(\alpha^N) = h\Bigl(\frac{-d(\alpha)}{c(\alpha^{-1})}\Bigr)
                 = h\Bigl(\frac{-\alpha^{\deg c} d(\alpha)}{\tilde{c}(\alpha)}\Bigr) \\
                &\le (\deg c + \deg d) h(\alpha) + \log\max\{s(c), s(d)\} \\
                &\le (\deg c + \deg d) h(\alpha) + \log(\|c\| + \|d\|) \,.
  \end{align*}
  Using that $h(\alpha) = (\log M(p))/(\deg p)$, this gives the first result.
  The remaining estimates follow from this and the lower bounds on~$M(p)$ mentioned
  above.
\end{proof}

Assume that $N > N_1$. Then
by Corollary~\ref{C:noncyc}, a non-cyclotomic irreducible factor~$p \in \Z[x]$ of~$f_N$ must
also be non-reciprocal. By~\cite{Smyth1971}, we have $M(p) \ge \theta_0$.
The multiplicativity of the Mahler measure and Landau's inequality $M(f) \le \sqrt{\|f\|}$
(see~\cite{HindrySilverman}*{Lemma~B.7.3.1~(iii)}) then give that
\[ \theta_0^n \le M(f_N) \le \sqrt{\|c\|+\|d\|} \,, \]
where $n$ is the
number of non-cyclotomic irreducible factors of~$f_N$. This shows that
\[ n \le \frac{\log(\|c\| + \|d\|)}{2 \log \theta_0} \,. \]

%==========================================================================

\section{Ljunggren's trick} \label{S:Lj}

As a motivation for our approach, we recall how Ljunggren deals
with the polynomials $x^n + \eps x + \eps'$ with $\eps, \eps' \in \{\pm1\}$.
(Actually, he considers general trinomials $x^n \pm x^m \pm 1$ and also
quadrinomials, but for our expository purposes, the special case is sufficient.
The result for $m = 1$ was
obtained earlier by Selmer~\cite{Selmer}, but with a different method.)

Let $R = \Z[x,x^{-1}]$ be the ring of Laurent polynomials with integral
coefficients. We note that its unit group is $R^\times = \{\pm x^n : n \in \Z\}$,
and we write $f \sim g$ when $f, g \in R$ are equal up to multiplication by
a unit. Note that $f \sim g$ implies that $f(x) f(x^{-1}) = g(x) g(x^{-1})$.
We will also make use of the fact that $\|f\|$ is the coefficient
of~$x^0$ in $f(x) f(x^{-1})$.

Let now $f_N(x) = x^N + \eps x + \eps'$ for some $N \ge 2$; then $\|f_N\| = 3$.
Assume that $f_N$ factors as $f_N(x) = g(x) h(x)$ with $g, h \in \Z[x]$ non-constant.
Set $G(x) \colonequals g(x) h(x^{-1}) \in R$. We obviously have that
\begin{equation}\label{E:ffGG}
  f_N(x) f_N(x^{-1}) = G(x) G(x^{-1})\,;
\end{equation}
in particular, $\|G\| = \|f_N\| = 3$.
So we can write $G(x) \sim x^{m} + \eta x^{k} + \eta'$ with $\eta, \eta' \in \{\pm1\}$.
Comparing coefficients in~\eqref{E:ffGG} then shows that $G(x) \sim f_N(x)$ or
$G(x) \sim f_N(x^{-1})$. Swapping $g$ and~$h$ if necessary, we can assume that
$G(x) \sim f_N(x)$; then $h(x) \mid_R G(x^{-1}) \sim f_N(x^{-1})$, which
implies that $h$ divides the reversed polynomial
$\tilde{f}_N(x) = x^N f_N(x^{-1}) = \eps' x^N + \eps x^{N-1} + 1$ in~$\Z[x]$.
We obtain that
\[ h(x) \mid \eps \tilde{f}_N(x) - \eps \eps' f_N(x) = x (x^{N-2} - \eps') \,. \]
So $h$ divides $\eps' f_N - \eps' x^2 (x^{N-2} - \eps') = x^2 + \eps \eps' x + 1$.

There are now two cases:
\begin{enumerate}[1.]\addtolength{\itemsep}{3pt}
  \item $\eps' = \eps$. Then $h(x) = x^2 + x + 1$. Let $\omega$ be a primitive
        cube root of unity. Then $h \mid f_N$ if and only if $f_N(\omega) = 0$.
        We have that $f_N(\omega) = \omega^{N \bmod 3} - \eps \omega^2$,
        which vanishes if and only if $\eps = 1$ and $n \equiv 2 \bmod 3$.
        We conclude that $x^N - x - 1$ is irreducible for all $N \ge 2$ and
        that $x^N + x + 1$ is irreducible when $N \not\equiv 2 \bmod 3$,
        whereas  $x^N + x + 1$ splits as $x^2 + x + 1$ times another irreducible factor
        when $N \equiv 2 \bmod 3$.
  \item $\eps' = -\eps$. Then $h(x) = x^2 - x + 1$ has roots $-\omega$, $-\omega^2$,
        so $h$ divides~$f_N$ if and only if $f_N(-\omega) = 0$. In this case,
        $f_N(-\omega) = (-1)^N \omega^{N \bmod 3} + \eps \omega^2$.
        So we conclude that $x^N + x - 1$ is irreducible for $N \not\equiv 5 \bmod 6$,
        whereas $x^N - x + 1$ is irreducible for $N \not\equiv 2 \bmod 6$.
        When $N$ is in the excluded residue class mod~$6$, then the polynomial
        splits as $x^2 - x + 1$ times another irreducible factor.
\end{enumerate}

%==========================================================================

\section{The main result} \label{S:main}

We will now generalize this approach to families of polynomials
``with a large gap'': as in Section~\ref{S:intro},
we fix $c, d \in \Z[x]$
with $c(0), d(0) \neq 0$ and consider the polynomials
\[ f_N(x) = x^N c(x^{-1}) + d(x) \]
for $N > \deg c + \deg d$. In the special case considered in Section~\ref{S:Lj},
we had $c = 1$ and $d = \pm x \pm 1$.

The key part of Ljunggren's trick was the implication
\[ f_N(x) f_N(x^{-1}) = G(x) G(x^{-1})
     \quad\Longrightarrow\quad G(x) \sim f_N(x) \quad\text{or}\quad G(x) \sim f_N(x^{-1}) \,.
\]
We now show that for any fixed~$N$, this implication is in fact equivalent
to the statement of Theorem~\ref{T:Schinzel}, under suitable assumptions on $c$ and~$d$.

\begin{proposition} \label{P:equiv}
  Let $c, d \in \Z[x]$ with $c(0), d(0) \neq 0$ be such that $c \neq \pm d$.
  Then for each $N > \deg c + \deg d$, the following two statements
  are equivalent.
  \begin{enumerate}[\upshape (1)]\addtolength{\itemsep}{3pt}
    \item If $G \in R$ satisfies $G(x) G(x^{-1}) = f_N(x) f_N(x^{-1})$,
          then $G(x) \sim f_N(x)$ or $G(x) \sim f_N(x^{-1})$.
    \item The non-reciprocal part of~$f_N$ is irreducible.
  \end{enumerate}
\end{proposition}

\begin{proof}
  We note that the assumptions on $c$ and~$d$ imply that
  $f_N$ is not reciprocal.

  We first show that (1) implies~(2).
  Consider a factorization $f_N(x) = g(x) h(x)$.
  We set $G(x) = g(x) h(x^{-1})$ as in Ljunggren's trick.
  Then $G(x) G(x^{-1}) = f_N(x) f_N(x^{-1})$,
  so by~(1), we have that $G(x) \sim f_N(x)$ or $G(x) \sim f_N(x^{-1})$.
  By swapping the roles of $g$ and~$h$ if necessary, we can assume that
  we are in the first case. This implies that $g \tilde{h} = \pm f_N = \pm g h$,
  so that $\tilde{h} = \pm h$, and $h$ is reciprocal.
  We have therefore shown that in any
  factorization of~$f_N$, one factor is reciprocal.
  Now write $f_N = g_1 \cdots g_m h$, where $g_1, \ldots, g_m$ are the
  non-reciprocal irreducible factors and $h$ is the product of the reciprocal
  irreducible factors. Then $m \ge 1$, since $f_N$ is non-reciprocal.
  If $m = 1$, then $g_1$ is the non-reciprocal part of~$f_N$ and irreducible,
  so we are done. So assume now that $m \ge 2$. We show that $g_i g_j$
  must be reciprocal for all $1 \le i < j \le m$. It suffices to do this
  for $(i,j) = (1,2)$. Since $g_1$ is non-reciprocal, in the factorization
  $f_N = g_1 \cdot (g_2 \cdots g_m h)$, the second factor must be reciprocal.
  Since $g_2$ is non-reciprocal, $g_3 \cdots g_m h$ is then also non-reciprocal.
  So in the factorization $f_N = (g_1 g_2) \cdot (g_3 \cdots g_m h)$,
  now the first factor must be reciprocal, proving the claim made above.
  If $m = 2$, this implies that $f_N$ is reciprocal, a contradiction.
  If $m \ge 3$, then the fact that $g_1 g_2$, $g_2 g_3$ and $g_1 g_3$ are
  all reciprocal implies that $g_1 = \pm \tilde{g}_2 = \pm g_3 = \pm \tilde{g}_1$
  (if $h_1$, $h_2$ are irreducible in~$\Z[x]$ and non-reciprocal and $h_1 h_2$ is reciprocal,
  then $h_1 = \pm \tilde{h}_2$), contradicting that $g_1$ is non-reciprocal.
  So $m = 1$ is the only possibility.

  Now we show the converse. Assume that $G \in R$ satisfies
  $G(x) G(x^{-1}) = f_N(x) f_N(x^{-1})$. Replacing $G$ by $x^n G(x)$ for a
  suitable~$n \in \Z$, we can assume without loss of generality that
  $G \in \Z[x]$ and $G(0) \neq 0$. Then $\deg G = N$, and we have that
  $G(x) \tilde{G}(x) = f_N(x) \tilde{f}_N(x)$. Comparing the factorizations
  of both sides into irreducibles, we see that there is a factorization
  $f_N = g h$ such that $G = \pm g \tilde{h}$. Since $f_N$ is not reciprocal,
  at least one of $g$ and~$h$ must
  be non-reciprocal as well; we can assume that this is~$g$ (otherwise
  we replace $G$ by~$\tilde{G}$). Since by~(2), the non-reciprocal part
  of~$f_N$ is irreducible, $g$ must contain the unique non-reciprocal
  irreducible factor of~$f_N$,
  hence $h$ is reciprocal. But then $\tilde{h} = \pm h$, and so $G = \pm f_N$.
\end{proof}

The idea is now to obtain a better lower bound for~$N$ such that statement~(1)
above holds,
which will then lead to the same better bound in Theorem~\ref{T:Schinzel}.
We need a condition on $c$ and~$d$ for this to work.

\begin{definition} \label{D:robust}
  A pair $(c,d)$, where $c, d \in \Z[x]$ with $c(0), d(0) \neq 0$, is
  \emph{weakly robust},
  if for each further pair of polynomials $a, b \in \Z[x]$ such that
  $a b = c d$, it follows that $\|a\| + \|b\| \ge \|c\| + \|d\| - 1$.
  The pair $(c,d)$ is \emph{robust} if for all $(a,b)$ as above, the additional
  relation $a(x) a(x^{-1}) + b(x) b(x^{-1}) = c(x) c(x^{-1}) + d(x) d(x^{-1})$
  implies that $(a,b) = \pm(c,d)$ or~$\pm(d,c)$,
  and whenever $\|a\| + \|b\| = \|c\| + \|d\| - 1$,
  we have $a + b \neq 0$. (This last condition can be relaxed; compare
  the proof of Lemma~\ref{L:trick}.)
\end{definition}

We note that the relation
$a(x) a(x^{-1}) + b(x) b(x^{-1}) = c(x) c(x^{-1}) + d(x) d(x^{-1})$
implies that $\|a\| + \|b\| = \|c\| + \|d\|$; this is simply the equality
of the coefficients of~$x^0$ on both sides.

For example, the pair $(c,d)$ is robust when $c = 1$ and $d$ is primitive
and irreducible or when $c$ and~$d$ are both primitive
and irreducible and $\|cd\| \ge \|c\| + \|d\| - 2$.

If $f \in \Z[x]$ and $m \in \Z_{\ge 0}$, we write $f|_m$ for $f$ truncated
to degree~$< m$ (i.e., $f|_m$ is the remainder when dividing $f$ by~$x^m$).

\begin{lemma} \label{L:tree1}
  Let $a, b, f \in \Z[x]$ with $\deg a, \deg b < 2m$, $\deg f < m$, $f(0) \neq 0$
  and $(ab)|_{2m} = f$. Then one of the following is true.
  \begin{enumerate}[\upshape(1)]
    \item $ab = f$;
    \item \label{L:tree1:2}
          $ab \neq f$, $a|_m \cdot b|_m \neq f$ and $\|a|_m\| + \|b|_m\| \le \|a\| + \|b\| - 1$;
    \item \label{L:tree1:3}
          $ab \neq f$, $a|_m \cdot b|_m = f$ and $\|a|_m\| + \|b|_m\| \le \|a\| + \|b\| - 2$.
  \end{enumerate}
\end{lemma}

\begin{proof}
  We can assume that $ab \neq f$. We write $a = \au + \ao x^m$, $b = \bu + \bo x^m$,
  where $\au = a|_m$ and $\bu = b|_m$. Then
  \[ ab = \au\,\bu + (\au\,\bo + \ao\,\bu) x^m + \ao\,\bo x^{2m}
        \equiv f \bmod x^{2m} \,.
  \]
  If $\au\,\bu \neq f$, then $\au\,\bo + \ao\,\bu \neq 0$,
  so $\ao \neq 0$ or $\bo \neq 0$,
  which implies that $\|\ao\| + \|\bo\| \ge 1$ and so gives~\eqref{L:tree1:2}.
  If $\au\,\bu = f$, then $\au\,\bo + \ao\,\bu \equiv 0 \bmod x^m$. In this case,
  $\ao = \bo = 0$ is not possible, since $a b \neq f$. Since $\au(0), \bu(0) \neq 0$,
  it then follows that $\ao \neq 0$ and $\bo \neq 0$, which implies
  that $\|\ao\| + \|\bo\| \ge 2$ and so gives~\eqref{L:tree1:3}.
\end{proof}

\begin{corollary} \label{C:tree}
  Assume that $(c,d)$ is weakly robust. Then there is
  \[ m_0 \le (1 + \deg c + \deg d) 2^{\|c\| + \|d\| - 1} \]
  such that
  for all $m > m_0$, if $a, b \in \Z[x]$ with $\deg a, \deg b < m$ satisfy $(ab)|_m = cd$,
  then
  \[ ab = cd \qquad\text{or}\qquad \|a\| + \|b\| > \|c\| + \|d\| \,. \]
\end{corollary}

\begin{proof}
  Consider $m > (1 + \deg c + \deg d) 2^{\|c\|+\|d\|-1}$. Assume there are $a, b \in \Z[x]$
  of degree less than~$m$ with $(ab)|_m = cd$, but $ab \neq cd$ and $\|a\| + \|b\| \le \|c\| + \|d\|$.
  By iteratively applying Lemma~\ref{L:tree1}, we either find that
  $\|a|_{1 + \deg c + \deg d}\| + \|b|_{1 + \deg c + \deg d}\| \le 1$,
  which is absurd, as $a(0), b(0) \neq 0$,
  or else that there is some $m' < m$ such that $a|_{m'} \cdot b|_{m'} = f$ and
  $\|a|_{m'}\| + \|b|_{m'}\| \le \|c\| + \|d\| - 2$, which contradicts the
  weak robustness of~$(c,d)$.
\end{proof}

We note that for any given pair~$(c,d)$, we can effectively determine the optimal
bound~$m_0$ by successively computing the sets
\begin{equation} \label{E:Tm}
  T_m = \{(a,b) \mid a, b \in \Z[x], \; \deg a, \deg b < m, \; a b \equiv c d \bmod x^m, \;
                     \|a\| + \|b\| \le \|c\| + \|d\|\}
\end{equation}
for $m = 1, 2, \ldots$,
until $ab = cd$ for all $(a,b) \in T_m$. Since forgetting the $x^m$~term gives
a natural map $T_{m+1} \to T_m$, it is easy to construct~$T_{m+1}$ from~$T_m$.
In Section~\ref{S:algo}, we explain how this approach can be improved to give a
more efficient procedure. In Section~\ref{S:m0}, we discuss the dependence
of~$m_0$ on $c$ and~$d$ in more detail.

We can now get our better bound for robust pairs.

\begin{lemma} \label{L:trick}
  Assume that $(c,d)$ is a robust pair.
  Then statement~(1) in Proposition~\ref{P:equiv} holds for all $N > N_0$, where
  \[ N_0 \le (1 + \deg c + \deg d) 2^{\|c\| + \|d\|} \,. \]
\end{lemma}

\begin{proof}
  As in the proof of Proposition~\ref{P:equiv}, we
  can assume that $G \in \Z[x]$ with $G(0) \neq 0$. Then we can write
  $G(x) = x^N a(x^{-1}) + b(x)$ with $a, b \in \Z[x]$, $\deg a < \lceil (N+1)/2 \rceil$,
  $\deg b < \lceil N/2 \rceil$, and we have that
  \begin{align}
    c(x) d(x) &+ x^N \bigl(c(x) c(x^{-1}) + d(x) d(x^{-1})\bigr) + x^{2N} c(x^{-1}) d(x^{-1}) \nonumber \\
     &= x^N f_N(x) f_N(x^{-1})
      = x^N G(x) G(x^{-1}) \label{E:abcd} \\
     &= a(x) b(x) + x^N \bigl(a(x) a(x^{-1}) + b(x) b(x^{-1})\bigr) + x^{2N} a(x^{-1}) b(x^{-1}) \,. \nonumber
  \end{align}
  We assume that $\lceil N/2 \rceil > \max\{\deg c, \deg d\}$; this implies that
  \[ a(x) b(x) \equiv c(x) d(x) \bmod x^{\lceil N/2 \rceil} \,. \]
  Now we assume in addition that $\lceil N/2 \rceil > m_0$, where $m_0$ is as in Corollary~\ref{C:tree}.
  Note that
  \[ \|c\| + \|d\| = \|f_N\| = \|G\| = \|a\| + \|b\| \ge \|a|_{\lceil N/2 \rceil}\| + \|b\| \,. \]
  By Corollary~\ref{C:tree}, it follows that $a|_{\lceil N/2 \rceil} b = c d$.
  When $N$ is odd, then $a|_{\lceil N/2 \rceil} = a$, so $a b = c d$.
  Comparing the expressions in~\eqref{E:abcd}, we see that we also must have that
  \[ a(x) a(x^{-1}) + b(x) b(x^{-1}) = c(x) c(x^{-1}) + d(x) d(x^{-1}) \,, \]
  and so by robustness of~$(c,d)$, it follows that $(a,b) = \pm(c,d)$ or $\pm(d,c)$,
  which is equivalent to $G = \pm f_N$ or $G = \pm \tilde{f}_N$, proving the claim.
  When $N$ is even, then either $\deg a < N/2$ and we can conclude
  as for $N$~odd. Or else (again by robustness)
  $\|a|_{N/2}\| + \|b\| = \|c\| + \|d\| - 1$ and so
  $a = a|_{N/2} \pm x^{N/2}$. We change notation and write $a$ for what
  was $a|_{N/2}$, so that now
  \[ G(x) = x^N a(x^{-1}) \pm x^{N/2} + b(x) \qquad\text{and}\qquad a b = c d \,. \]
  In this case, we need to assume that $N/4 > \deg c + \deg d$ (note that this case is only
  possible when there is a factorization $cd = ab$ with $\|a\| + \|b\| = \|c\| + \|d\| - 1$,
  which we can check beforehand). Then, comparing the expressions again,
  we see that $a + b = 0$ (this uses that
  $\deg a, \deg b \le \deg a + \deg b = \deg c + \deg d < N/4$), which contradicts robustness.
  So this case cannot occur. (Comparing the ``middle part'' of the two
  polynomials gives the additional relation
  $2 a(x) a(x^{-1}) + 1 = c(x) c(x^{-1}) + d(x) d(x^{-1})$, so it would be sufficient
  to require that this cannot hold when $cd = -a^2$.)

  This shows the claim with
  \[ N_0 = 2 \max\{\deg c, \deg d, m_0\} \]
  when there is no pair~$(a,b)$ with $ab = cd$ and $\|a\| + \|b\| < \|c\| + \|d\|$,
  whereas in the other case, we can take
  \[ N_0 = \max\{4 \deg (cd), 2 m_0\} \,. \]
  Since $m_0 \le (1 + \deg c + \deg d) 2^{\|c\| + \|d\| - 1}$ by Corollary~\ref{C:tree}
  and $\|c\| + \|d\| \ge 2$, we can take $N_0 = (1 + \deg c + \deg d) 2^{\|c\| + \|d\|}$
  in all cases.
\end{proof}

Combining Lemma~\ref{L:trick} with Proposition~\ref{P:equiv} now
immediately gives Theorem~\ref{T:main}.

We note that we do not use the assumption that $(c,d)$ is non-Capellian
in our proof. Since the result excludes the existence of ``Capellian'' factorizations,
this implies that a robust pair~$(c,d)$ is necessarily non-Capellian.

\begin{remark}
  An alternative proof of Theorem~\ref{T:Schinzel} via Proposition~\ref{P:equiv}
  can be obtained from work of Bombieri and Zannier on unlikely intersections.
  The relevant result can be found in~\cite{BMZ2007}*{Theorem~1.6}
  (or~\cite{SchinzelBook}*{Appendix by Zannier}). Iterating their result
  (with
  \[ P(x_1, x_2, x_3, \ldots, x_n) = x_2 c(x_1^{-1}) + d(x_1)
     \quad\text{and}\quad
     Q(x_1, x_2, x_3, \ldots, x_n) = a_3 x_3 + \ldots + a_n x_n \,,
  \]
  where $(a_3, \ldots, a_n)$ runs through all vectors of nonzero integers
  of weight at most~$w$; in our case $\zeta_j = 1$ for all~$j$)
  leads to the following statement.

  \emph{Let $c, d \in \Z[x]$ with $c(0), d(0) \neq 0$, $\gcd_{\Z[x]}(\tilde{c}, d) = 1$
  and $(c,d)$ not Capellian.
  For any $w > 0$ there is $N_{\text{\upshape BZ}}(c, d, w)$ with the following
  property. If $N > N_{\text{\upshape BZ}}(c, d, w)$ and $g \in R$ has weight
  $\|g\| \le w$, then either $f_N$ divides~$g$ in~$R$, or else the gcd of $f_N$ and~$g$
  is a product of cyclotomic polynomials.}

  From this, it is easy to conclude that statement~(1) in Proposition~\ref{P:equiv}
  holds as soon as $N > N_{\text{\upshape BZ}}(c, d, \|c\|+\|d\|)$ when $c$ and~$d$ satisfy
  the assumptions in Theorem~\ref{T:Schinzel}. This proves Theorem~\ref{T:Schinzel}
  with $N_0 = N_{\text{\upshape BZ}}(c, d, \|c\|+\|d\|)$. The bound is effective,
  but is not made explicit in~\cite{BMZ2007}.
\end{remark}

%==========================================================================

\section{Growth of $m_0$} \label{S:m0}

We write $m_0(c,d)$ for the optimal value of~$m_0$ in Corollary~\ref{C:tree}.
The bound
\[ m_0(c,d) \le (1 + \deg c + \deg d) 2^{\|c\|+\|d\|-1} \]
obtained in Corollary~\ref{C:tree} is exponential in the weight of $c$ and~$d$.
At least in some cases, we can do better.

\begin{lemma} \label{L:linear}
  Let $k, l \in \Z$ and set $c = 1 + kx$, $d = 1 + lx$. Note that $(c,d)$ is robust.
  \begin{enumerate}[\upshape(1)]\addtolength{\itemsep}{3pt}
    \item If $|k-l| \ge 6$, then $m_0(1+kx, 1+lx) = 2$.
    \item If $2 \le |k-l| \le 5$ and $\max\{|k|, |l|\} \ge 3$, then $m_0(1+kx, 1+lx) = 3$.
    \item If $|k-l| \le 1$, then $m_0(1+kx, 1+lx) = 1$.
  \end{enumerate}
  The exceptional cases are given by
  \[ m_0(1, 1+2x) = 4\,, \; m_0(1+x,1-x) = 2\,, \; m_0(1+2x,1-x) = 5\,, \;
     m_0(1+2x,1-2x) = 4\,,
  \]
  together with $m_0(1+lx,1+kx) = m_0(1+kx,1+lx) = m_0(1-kx,1-lx)$.
\end{lemma}

\begin{proof}
  We can assume that $k \le l$. We write (without loss of generality)
  \[ a = 1 + a_1 x + a_2 x^2 + a_3 x^3 + \ldots \qquad\text{and}\qquad
     b = 1 + b_1 x + b_2 x^2 + b_3 x^3 + \ldots
  \]
  with $a_1 \le b_1$. The condition $ab \equiv cd \bmod x^m$ is then
  \[ a_1 + b_1 = k + l\,, \quad
     a_2 + a_1 b_1 + b_2 = k l\,, \quad
     a_3 + a_2 b_1 + a_1 b_2 + b_3 = 0\,, \quad
     \ldots \,;
  \]
  we look at the first~$m-1$ equations. The condition $\|a\| + \|b\| \le \|c\| + \|d\|$ is
  \[ a_1^2 + b_1^2 + a_2^2 + b_2^2 + a_3^2 + b_3^2 + \ldots \le k^2 + l^2 \,. \]
  Write $a_1 = k + \alpha$; then $b_1 = l - \alpha$, $\alpha \le (l-k)/2$, and
  \[ k^2 + l^2 \ge a_1^2 + b_1^2 = k^2 + l^2 - 2\alpha(l-k-\alpha) \,, \]
  which implies that $\alpha \ge 0$. If $\alpha = 0$, then $(a, b) = (c, d)$,
  and we are done. So we now assume that $\alpha \ge 1$. Since $\alpha \le (l-k)/2$,
  this is not possible when $l-k \le 1$, which proves case~(3).

  When $m \ge 3$, then the coefficient of~$x^2$ gives us that
  \[ a_2 + b_2 = k l - a_1 b_1 = -\alpha(l-k-\alpha) \,, \]
  which implies that
  \[ k^2 + l^2 - (a_1^2 + b_1^2) = 2\alpha(l-k-\alpha)
      \ge a_2^2 + b_2^2 \ge \frac{\alpha^2(l-k-\alpha)^2}{2}
  \]
  and therefore that
  \[ \alpha(l-k-\alpha) \le 4 \,. \]
  This is impossible when $l-k \ge 6$, which proves case~(1).

  The remaining cases are
  \[ (\alpha,l-k) = (1,2)\,, \; (1,3)\,, \; (1,4)\,, \; (1,5)\,, \; (2,4)\,. \]
  Then $a_2 + b_2 = -1$, $-2$, $-3$, $-4$, $-4$, and
  $\|a_2\| + \|b_2\| \le 2$, $4$, $6$, $8$, $8$, respectively.
  In the cases $(1,5)$ and~$(2,4)$, we must have that
  $a_2 = b_2 = -2$, which implies that $a_3 = b_3 = 0$; this leads to a contradiction
  for $m \ge 4$ unless $(k,l) = (-2,2)$. In the case~$(1,3)$ with $\{a_2,b_2\} = \{0,-2\}$,
  we obtain a similar contradiction unless $k \in \{-2,-1\}$.
  We consider the remaining cases in turn.
  \begin{itemize}\addtolength{\itemsep}{3pt}
    \item $l = k+2$, $a_1 = k+1 = b_1$, $a_2 = -1$, $b_2 = 0$. \\
          This gives that $\|a_3\| + \|b_3\| \le 1$ and $a_3 + b_3 = k+1$, which is
          impossible unless $k \in \{-2,-1,0\}$.
    \item $l = k+3$, $a_1 = k+1$, $b_1 = k+2$, $a_2 = b_2 = -1$. \\
          This gives that $\|a_3\| + \|b_3\| \le 2$ and $a_3 + b_3 = 2k+3$, which is
          impossible unless $k \in \{-2,-1\}$.
    \item $l = k+4$, $a_1 = k+1$, $b_1 = k+3$, $\{a_2,b_2\} = \{-1,-2\}$. \\
          This gives that $\|a_3\| + \|b_3\| \le 1$ and $a_3 + b_3 = 3k+5$ or~$3k+7$, which is
          impossible unless $k = -2$.
  \end{itemize}
  This proves case~(2). The exceptional values can be determined by the algorithm
  sketched after the proof of Corollary~\ref{C:tree}.
\end{proof}

We can also deal with $(c,d) = (1+kx, l)$. Such a pair (with $|l| > 1$ and $k \neq 0$;
note that $|l| = 1$ is covered by Lemma~\ref{L:linear}) is robust if and only if
$|k| \ge |l|/p$, where $p$ is the smallest prime divisor of~$l$. We give a result
for slightly larger~$|k|$. Note that by changing the sign of~$x$ or of $d = l$,
we can assume without loss of generality that $k, l > 0$.

\begin{lemma} \label{L:linear2}
  Let $l \ge 5$; let $k \ge 1$ if $l$ is prime and $k > \sqrt{p^2 + l^2/p^2}$ otherwise,
  where $p$ is the smallest prime divisor of~$l$. If $k \le l^2/2$, then $m_0(1+kx, l) = 1$.
\end{lemma}

\begin{proof}
  We have to show that the only pair of polynomials $a = a_0 + a_1 x$, $b = b_0 + b_1 x$
  such that $1 \le a_0 \le b_0$, $\|a\| + \|b\| \le 1 + k^2 + l^2$
  and $a b \equiv l + k l x \bmod x^2$ is $(a,b) = (1+kx, l)$.
  The last condition is equivalent to the pair of equations
  \[ a_0 b_0 = l \qquad\text{and}\qquad a_0 b_1 + a_1 b_0 = k l \,. \]
  If $a_0 = 1$ and therefore $b_0 = l$, then we have to solve
  $b_1 + l a_1 = k l$ under the condition that $a_1^2 + b_1^2 \le k^2$.
  The equation implies that $l$ divides~$b_1$. Writing $b_1 = l \beta$,
  we have that $a_1 = k - \beta$ and $(k - \beta)^2 + l^2 \beta^2 \le k^2 \le (l+1)^2$.
  Since $l \ge 5$, $4 l^2 > (l+1)^2$, which implies that $\beta \in \{-1,0,1\}$.
  If $\beta = 0$, then $(a,b) = (1+kx, l)$ as desired. In the other cases,
  we obtain that $l^2 \le 2k - 1 \le l^2 - 1$, a contradiction.
  If $l$ is prime, then $(1, l)$ is the only possibility for~$(a_0, b_0)$, so in this
  case, we are done.

  Now assume that $1 < a_0 \le b_0$; then $a_0 \ge p$. It is easy to see
  (using the Cauchy-Schwarz inequality, for example) that $a_0 b_1 + a_1 b_0 = k l$
  implies that
  \[ a_1^2 + b_1^2 \ge \frac{k^2 l^2}{a_0^2 + b_0^2} \,. \]
  From
  \[ \frac{k^2 l^2}{a_0^2 + b_0^2} \le a_1^2 + b_1^2 \le 1 + k^2 + l^2 - (a_0^2 + b_0^2) \]
  we conclude that
  \[ k^2 \le (a_0^2 + b_0^2) \frac{1 + l^2 - (a_0^2 + b_0^2)}{l^2 - (a_0^2 - b_0^2)}
         = a_0^2 + b_0^2 + \frac{a_0^2 + b_0^2}{l^2 - (a_0^2 + b_0^2)} \,.
  \]
  Since $p \ge 2$, we have that $a_0^2 + b_0^2 \le 4 + l^2/4$, which together
  with $l \ge 5$ implies that the last fraction is strictly less than~$1$.
  Since $k^2$ and $a_0^2 + b_0^2$ are both integers, we must have that
  $k^2 \le a_0^2 + b_0^2 \le p^2 + l^2/p^2$, contradicting our assumption on~$k$.
  So the case under consideration is impossible; this finishes the proof.
\end{proof}

\begin{remark} \label{R:linear2}
  For the robust pairs $(1+kx, l)$ with $l > 1$ that are not covered by Lemma~\ref{L:linear2},
  it appears that $m_0(1+kx, l) \le 2$, except for
  $m_0(1+2x, 4) = 4$ and $m_0(1+3x, 9) = 3$. It should be possible to prove
  this, but the arguments seem to get rather technical.
\end{remark}

\begin{remark} \label{Rmk}
  Experiments suggest that better bounds are likely to be true in other cases as well.
  \begin{enumerate}[(1)]\addtolength{\itemsep}{3pt}
    \item For $c = 1$, $d = 1 - kx^2$ with $k \ge 2$ not a square, $m_0(c,d)$ seems
          to grow at most linearly with~$k$:
          \[ \begin{array}{r|*{20}{c}}
               k   &  2 &  3 &  5 &  6 &  7 &  8 & 10 & 11 & 12 & 13 & 14 & 15 & 17
                   & 18 & 19 & 20 & 21 & 22 & 23 & 24 \\\hline
               m_0 &  8 & 11 & 23 & 23 & 20 & 29 & 34 & 37 & 39 & 44 & 46 & 48 & 54
                   & 69 & 71 & 66 & 66 & 58 & 59 & 76
             \end{array}
          \]
    \item It looks like $m_0(1, 1 + kx^2) = 4$ when $k \ge 6$.
    \item It appears that $m_0(1, 1 + kx - kx^2) = k^2-2k$ when $k \ge 3$ or $k \le -5$.
    \item For $c = 1$, $d = 1 + kx^3$, $k$ not a cube, $m_0(c,d)$ appears to grow linearly
          with~$k$ again:
          \[ \begin{array}{r|*{10}{c}}
               k   &  1 &  2 &  3 &  4 &  5 &  6 &  7 &  9 & 10 & 11 \\\hline
               m_0 &  3 & 12 & 13 & 14 & 16 & 24 & 26 & 26 & 31 & 35
             \end{array}
          \]
    \item We have that $m_0(1, 1 + 2x^k) = 4k$ for all $k \ge 1$.
          (For comparison, the bound from the proof above is $32(k + 1)$.)
    \item It looks like $m_0(1, 1 + 3x^k) = 3 k$ when $3 \nmid k$
          (with exceptions for $k = 4$ and $k = 8$) and $m_0(1, 1 + 3x^k) = 13 l$ when $k = 3 l$.
    \item It appears that the growth of $m_0(1, 1 + k x^k)$ is quadratic in~$k$:
          \[ \begin{array}{r|*{8}{c}}
               k   &  1 &  2 &  3 &  4 &  5 &  6 &  7 &  8 \\\hline
               m_0 &  1 &  8 & 13 & 16 & 26 & 48 & 37 & 66
             \end{array}
          \]
  \end{enumerate}
\end{remark}

Heuristically, we expect that the weight of any pair has to grow after
increasing $m$ by a bounded amount, which would translate into a bound
that is linear in $\|c\| + \|d\|$. However, it turns out that this is wrong.
We define, for $\delta \in \Z_{\ge 1}$ and $\alpha > 0$,
\[ \mu_\alpha(\delta) = \limsup_{\|c\|+\|d\| \to \infty} \frac{m_0(c,d)}{\delta (\|c\|+\|d\|)^\alpha} \,, \]
where $(c,d)$ runs over all robust pairs with $\deg(cd) = \delta$.
We write $T_m(c,d)$ for the set~$T_m$ defined in~\eqref{E:Tm} for the pair~$(c,d)$.

We note that it is easy to see that $m_0(c(x^n), d(x^n)) \ge n m_0(c,d)$
(if $(a,b) \in T_m(c,d)$, then $(a(x^n), b(x^n)) \in T_{nm}(c(x^n), d(x^n))$),
which implies that $\mu_\alpha(n \delta) \ge \mu_\alpha(\delta)$.

\begin{proposition} \label{P:limits}
  \[ \mu_1(1) \ge \frac{1}{2}\,, \qquad \mu_2(2) \ge \frac{1}{144}\,, \qquad
     \mu_2(3) \ge \frac{3}{200}\,, \qquad \mu_2(4) \ge \frac{25}{1568}\,,
  \]
  and
  \[ \liminf_{k \to \infty} \mu_2(k) \ge \frac{25}{1568} \,. \]
\end{proposition}

\begin{proof}
  For $\delta = 1$, consider $c = 1$ and $d = (k+1) - kx$ for $k \in \Z \setminus \{-1, 0\}$.
  Since $d$ is primitive, $(c,d)$ is clearly robust. We have that
  \[ (a, b) := \Bigl(1-x, \frac{1-x^{k^2}}{1-x} + k\Bigr) \in T_{k^2}\bigl(1, (k+1) - kx\bigr) \]
  (note that $\|a\| + \|b\| = 2 + (k+1)^2 + k^2-1 = 2k^2 + 2k + 2 = \|c\| + \|d\|$),
  which shows that $m_0(1, d) \ge k^2$. This implies that
  \[ \mu_1(1) \ge \lim_{|k| \to \infty} \frac{k^2}{1 \cdot (2k^2 + 2k + 2)} = \frac{1}{2} \,. \]

  For $\delta \ge 2$, we consider the following general construction.
  Take positive integers $L$, $M$ and~$N$ and fix a factorization $1-x^L = \Phi \Psi$
  with $\Phi(0) =\Psi(0) = 1$. We set
  \[ t = \Phi \frac{1 - x^{LMN}}{1 - x^{LM}} \qquad\text{and}\qquad
     u = \Psi \frac{1 - x^{LM}}{1 - x^L} \,;
  \]
  then $t u = 1 - x^{LMN} \equiv 1 \bmod x^{LMN}$. We further fix a primitive polynomial
  $p \in \Z[x]$ such that $\deg (\Phi p) = \delta$. Then for $k \in \Z \setminus \{0\}$,
  \[ t \bigl(u + k p (1 - x^{LM})\bigr) \equiv  1 + k p \Phi \bmod x^{LMN} \]
  and
  \begin{align*}
    \|t\| + \|u + k p (1 - x^{LM})\|
      &= N \|\Phi\| + \|u + kp\| + k^2\|p\| \\
      &= N \|\Phi\| + M \|\Psi\| + 2 k^2 \|p\| + O(k) \,.
  \end{align*}
  Consider $c = 1$ and $d = 1 + k p \Phi$.
  We note that $d$ is primitive when $k \in g \Z$, where
  $g$ is the content of $p - p(0)$. We can then choose $k = \ell g$, where $\ell$
  is a  prime not dividing the leading coefficient of~$p$; then $d$ is irreducible
  by the Eisenstein criterion applied to~$\tilde{d}$.
  Also, $\|c\| + \|d\| = \|1\| + \|1 + k p \Phi\| = k^2 \|p \Phi\| + O(k)$.
  So we can choose $M$ and~$N$ satisfying
  \[ N \|\Phi\| + M \|\Psi\| \le k^2 \bigl(\|p \Phi\| - 2 \|p\|\bigr) - O(k) \,; \]
  then $(t, u + k p (1 - x^{LM})) \in T_{LMN}(c, d)$ and $m_0(c,d) \ge LMN$.
  The maximal value of~$LMN$ is obtained when
  \[ N \|\Phi\| \approx M \|\Psi\| \approx \frac{\|p \Phi\| - 2 \|p\|}{2} k^2 \,, \]
  which gives
  \[ m_0(c,d) \ge LMN \approx \frac{L \bigl(\|p \Phi\| - 2 \|p\|\bigr)^2}{4 \|\Phi\| \|\Psi\|} k^4 \]
  and so, letting $|k| \to \infty$ (through values such that $d$ is primitive and irreducible),
  \[ \mu_2(\delta)
      \ge \frac{L}{4 \delta \|\Phi\| \|\Psi\|} \Bigl(1 - 2 \frac{\|p\|}{\|p \Phi\|}\Bigr)^2 \,.
  \]
  For $\delta = 2$, we take $\Phi = 1-x$, $L = 1$, $p = 1-x$ (or $\Phi = 1+x+x^2$, $L = 3$, $p = 1$),
  which gives the bound $1/144$. For $\delta = 3$, we take $\Phi = 1+x+x^2$, $L = 3$, $p = 1+x$,
  which gives the bound $3/200$. For $\delta = 4$, we take $\Phi = 1+x+x^2+x^3$, $L = 4$, $p = 1+x$,
  which gives the bound $25/1568$.

  For $\delta = 4\nu$, we have $\mu_2(\delta) \ge \mu_2(4)$ by the discussion
  above. For $\delta = 4\nu + j$ with $\nu \ge 1$ and $j \in \{1,2,3\}$,
  we use $\Phi = 1 + x^\nu + x^{2\nu} + x^{3 \nu}$, $L = 4 \nu$, $p = x^j (1 + x^\nu)$,
  which gives that $\mu_2(\delta) \ge \frac{4 \nu}{4 \nu + j} \frac{25}{1568}$;
  this implies the last claim.
\end{proof}

\begin{remark} \label{R:opti}
  Computations indicate that the pairs $(\Phi, p)$ we have chosen in the
  proof give the optimal limit value for degrees up to~$4$. Furthermore,
  it appears that $(1+x+x^2+x^3, 1+x)$ leads to the overall maximal limit value.
  We have not attempted to prove this, though.
\end{remark}

\begin{figure}[htb]
  \begin{center}
    \Gr{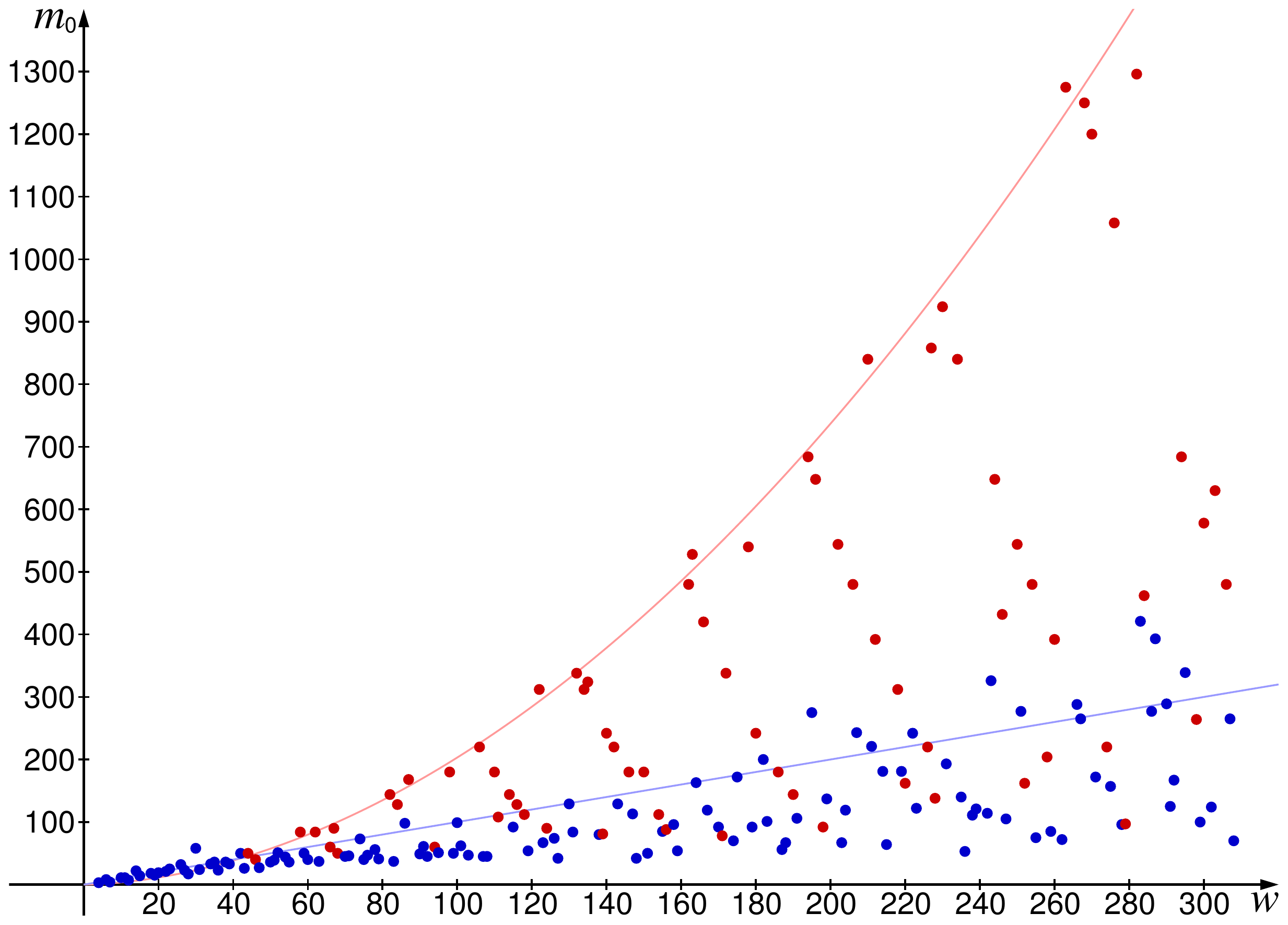}{0.75\textwidth}
  \end{center}
  \caption{Maximal values of $m_0(1,d)$ for robust pairs $(1,d)$
           of given weight $\|1\| + \|d\| = w$ and with $\deg d = 2$.
           Red dots indicate values obtained by pairs covered by the
           construction in the proof of Proposition~\ref{P:limits}.
           The red curve is $m_0 = \frac{1}{72} w^2 + \frac{\sqrt{3}}{27} w^{3/2}$,
           the blue curve is $m_0 = w$.}
  \label{fig0}
\end{figure}

Experimental evidence suggests that in the degree~$1$ case, the pairs
achieving a new maximal~$m_0$ (for all pairs of no larger weight) are
indeed of the form given in the proof (starting from weight~$86$), and
the pair~$(a,b)$ giving the maximum is also as given in the proof.
For degrees $2$, $3$ and~$4$, from some point on, the pairs giving
large values of~$m_0$ are also of the same general form as those in the proof
above (or slight variations thereof, where $u$ is multiplied by some
small degree polynomial~$q$). This is illustrated for degree~$2$ in Figure~\ref{fig0}.

This is a strong indication that the construction
in the proof is essentially optimal, which, in conjunction with Remark~\ref{R:opti},
leads us to propose the following conjecture.

\begin{conjecture}
  \[ \mu_1(1) = \frac{1}{2}\,, \qquad \mu_2(2) = \frac{1}{144}\,, \qquad
     \mu_2(3) = \frac{3}{200}\,, \qquad \mu_2(4) = \frac{25}{1568}\,,
  \]
  and
  \[ \lim_{k \to \infty} \mu_2(k) = \frac{25}{1568} \,. \]
\end{conjecture}

This would imply that for any fixed degree~$\delta$,
there is a constant~$C_\delta$ such that for all robust pairs $(c,d)$
with $\deg(cd) = \delta$, we have the following upper bound:
\[ m_0(c, d) \le C_\delta (\|c\| + \|d\|)^2 \,. \]
It appears likely that this can be strengthened to
\[ m_0(c, d) \le C (\deg c + \deg d) (\|c\| + \|d\|)^2 \]
with a constant~$C$.
Feeding this into the proof of Theorem~\ref{T:main}, this would imply
that we can improve the bound for~$N_0$ to
\[ N_0 \le (\deg c + \deg d) \max\{4, 2 C (\|c\| + \|d\|)^2\} \,. \]

%==========================================================================

\section{An improvement of the algorithm determining $m_0$} \label{S:algo}

When $\|c\| + \|d\|$ is not very small, the sets~$T_m$ defined in~\eqref{E:Tm}
can get rather large, which makes the algorithm sketched after the proof
of Corollary~\ref{C:tree} rather slow. Here we describe how we can reduce
the size of the sets we have to consider, which gives a more efficient algorithm.

In a first step, we determine a lower bound~$\mu$ on~$m_0$ by reducing the sets~$T_m$
to the (say) $1000$ pairs~$(a,b)$ with smallest weight $\|a\| + \|b\|$
(if $\#T_m > 1000$) before computing~$T_{m+1}$ from~$T_m$.
The heuristic here is that the pairs of smallest weight
have the best chance of producing a pair for large~$m$.

The second step is then to construct the sets~$T_m$ successively as in
the original algorithm, but to prune them of as many pairs as possible
without changing the final result, with the goal to keep the sets small.
We can remove a pair~$(a,b)$ from~$T_m$ when $m > \deg (cd)$ and we
can show that any extension $(a + a_1 x^m, b + b_1 x^m)$ with
\[ \deg a_1, \deg b_1 < m_1 = \min\{m, \mu-m\} \quad\text{and}\quad
   (a + a_1 x^m)(b + b_1 x^m) \equiv c d \bmod x^{m+m_1}
\]
would have
$\|a\| + \|a_1\| + \|b\| + \|b_1\| > \|c\| + \|d\|$. We can get a lower
bound for $\|a_1\| + \|b_1\|$ if we allow real instead of integral coefficients.
Writing $a b = c d + h x^m$, we have to solve the following linear system
in the coefficients of $a_1$ and~$b_1$:
\[ b a_1 + a b_1 \equiv -h \bmod x^{m_1} \,. \]
If $M$ is the matrix such that this system is $(a_1,b_1) M = -h$
(where we identify $a_1$, $b_1$ and~$h$ with their coefficient vectors),
then the minimum is given by the squared Euclidean length~$\eta$ of $h M^+$,
where $M^+ = (M^\top M)^{-1} M^\top$ is the pseudoinverse of~$M$.
So we compute~$\eta$, and if $\eta > \|c\|+\|d\| - (\|a\|+\|b\|)$,
then we know that $(a,b)$ does not have any descendents in~$T_{m+m_1}$
and so will not lead to a better lower bound on~$m_0$ than~$\mu$.
We can therefore discard~$(a,b)$ in this case. This leads to considerably
smaller sets~$T_m$ than before.

An alternative approach (that has the advantage of requiring only a modest
amount of memory) is to use a best-first search that at each level always expands the
pair~$(a,b)$ of smallest weight that has not yet been considered.
As soon as we find a terminal node,
i.e., a pair~$(a,b)$ with $ab \neq cd$ that cannot be extended further,
we have a lower bound for~$m_0$, which we can use to prune the search tree
as described above. When we find another terminal node at a higher level~$m$,
then we update the lower bound.

%==========================================================================

\section{Some examples} \label{S:ex}

We present some applications of our algorithm. We begin with a
generalization of the family that figured in the MO question
mentioned in the introduction. In the following, $r \in \Z[x]$ is
the polynomial defined in Section~\ref{S:intro}; recall that it has
the property that any reciprocal factor of~$f_N$ divides~$r$.

\begin{example}
  Let $k$ be an integer with $|k| \ge 3$.
  Then there is an integer $N_k \ge 2$ such that for all $N > N_k$,
  the polynomial $x^N - k x^2 + 1$ is irreducible over~$\Q$.

  This follows by taking $c(x) = 1$ and $d(x) = 1 - k x^2$ in Corollary~\ref{C:noncyc}.
  Note that here $r = k (x^4 - k x^2 + 1)$, so (since $|k| > 2$)
  $r$ has no cyclotomic factors, implying that the non-cyclotomic part
  of~$f_N$ is $f_N$ itself.

  If in addition, $k$ is not a square, then
  $d$ is irreducible, so $(c,d)$ is robust, and our
  Theorem~\ref{T:main} applies. For the original
  question with $k = 7$, we find that $m_0(1, 1-7x^2) = 20$
  (see Remark~\ref{Rmk}), so we can take $N_7 = 40$.
  Checking smaller~$N$ separately, we find that $x^N - 7 x^2 + 1$
  is irreducible for all $N \ge 5$ and for $N = 3$, whereas
  \[ x^4 - 7 x^2 + 1 = (x^2 - 3 x + 1)(x^2 + 3 x + 1) \,. \]
  We similarly find that $x^N - k x^2 + 1$ is irreducible for $N \ge 5$
  when $3 \le k \le 24$ (and $k$ is not a square). Our implementation
  of the procedure that determines~$m_0$ gets quite slow beyond that point,
  but it is certainly tempting to conjecture that the statement remains
  true for larger~$k$.

  We remark that for $k \le -3$, we can easily
  show that the polynomial is irreducible for all $N \ge 3$: Comparing the
  polynomial with its dominant term and using Rouch\'e's theorem,
  we see that it has exactly two roots of absolute value $< 1$
  (and none of absolute value~$1$), which are complex conjugate,
  so would have to be roots of the same irreducible factor. But then
  any other factor would have all its roots of absolute value strictly larger
  than~$1$, which is impossible. (This is a variant of Perron's criterion;
  see Example~\ref{Ex3} below.)

  Note that for $0 < |k| \le 2$, we do indeed get arithmetic progressions
  of~$N$ such that \hbox{$x^N - k x^2 + 1$} is reducible:
  $N \equiv 4 \bmod 12$ for $k = 1$, $N \equiv 1 \bmod 3$ for $k = -1$,
  everything for $k = 2$, and $N \equiv 0 \bmod 4$ for $k = -2$.
\end{example}

\begin{example}
  We can also deal with $x^N - 4 x^2 + 1$, even though the pair $(1, 1-4x^2)$
  is \emph{not} robust. In this case, it is not hard to show that (up to
  a common sign change and order) for $m > 12$, the set~$T_m$ defined
  after the proof of Corollary~\ref{C:tree} consists of the pairs
  \begin{gather*}
    (1, 1 - 4x^2)\,, \qquad (1 + 2x, 1 - 2x)\,, \\
    (1 + 2x + x^{m-1}, 1  - 2x - x^{m-1})\,, \qquad
    (1 + 2x - x^{m-1}, 1  - 2x + x^{m-1})\,, \\
    (1 + 2x + 2x^{m-1}, 1  - 2x - 2x^{m-1}) \quad\text{and}\quad
    (1 + 2x - 2x^{m-1}, 1  - 2x + 2x^{m-1})\,.
  \end{gather*}
  Assume we are given~$G$ with $G(x) G(x^{-1}) \sim f_N(x) f_N(x^{-1})$;
  write $G \sim x^N a(x^{-1}) + b(x)$ as before.
  If $N$ is odd (and large), then
  $(a, b) \in T_{(N+1)/2}$ and $\|a\| + \|b\| = \|f_N\| = 18$. This forces
  (up to sign and order)
  \[ (a, b) = (1, 1 - 4 x^2) \qquad\text{or}\qquad
     (a, b) = (1 + 2x \pm 2 x^{(N-1)/2}, 1 - 2x \mp 2 x^{(N-1)/2}) \,.
  \]
  In the first case, $G \sim f_N$. In the second case,
  we easily see that $G(x) G(x^{-1}) \hspace{-1pt}\not\hspace{1pt}\sim f_N(x) f_N(x^{-1})$,
  for example by comparing coefficients of~$x^{(N-1)/2}$.
  If $N$ is even, then (up to sign and
  replacing by the reversed polynomial) $G(x) = x^N a(x^{-1}) + \gamma x^{N/2} + b(x)$
  with $(a, b) \in T_{N/2}$ satisfying $\|a\| + \|b\| + \gamma^2 = 18$. This rules out
  \[ (a, b) = (1 + 2x, 1 - 2x) \qquad\text{and}\qquad
     (a, b) = (1 + 2x \pm x^{N/2-1}, 1 - 2x \mp x^{N/2-1})\,,
  \]
  since $\gamma^2$ would have to be $8$ or~$6$. Then $\gamma = 0$,
  and we have essentially the same two cases as for odd~$N$, with the slight
  difference that $G(x) = x^N + 2 x^{N-1} \pm 2 x^{N/2+1} \mp 2 x^{N/2-1} - 2 x + 1$
  has a small gap between the middle two terms in the ``bad'' case;
  we still obtain a contradiction, though.

  For families of the form $x^N - k^2 x + 1$ with $k \ge 3$, the size
  of the sets~$T_m$ does not stabilize as for $k = 2$, but grows fairly
  quickly, so a simple analysis like the one above is no longer possible.
  It may still be true, however, that for large enough~$m$, these sets
  can be described using finitely many ``patterns'', which might make
  the situation amenable to a similar analysis.
\end{example}

\begin{example} \label{Ex3}
  We fix $k, l \in \Z$ with $k \neq l$ and consider
  the polynomial
  \[ f_N = x^N + k x^{N-1} + l x + 1 \qquad\text{for}\quad N \ge 3 \,. \]
  (When $k = l$, then $f_N$ is reciprocal, so that Theorem~\ref{T:Schinzel}
  does not apply.)
  We recall \emph{Perron's irreducibility criterion}~\cite{Perron}*{Theorem~I}
  (or~\cite{Prasolov}*{Theorem~2.2.5}), which says that a polynomial
  \[ f = x^N + a_{N-1} x^{N-1} + \ldots + a_0 \in \Z[x] \qquad\text{with}\quad a_0 \neq 0 \]
  is irreducible when $|a_{N-1}| >  1 + |a_0| + \ldots + |a_{N-2}|$.
  This shows that $f_N$ is irreducible for all $N \ge 3$ whenever $\bigl||k|-|l|\bigr| \ge 3$.
  (When $|l|$ is the larger absolute value, then we apply the criterion to~$\tilde{f}_N$.)

  This leaves (up to symmetry) the cases $l = k+1$, $l = k+2$ and $-2 \le k + l \le 2$.
  We note that
  \[ r(x) = (k-l)\bigl(x^2 + (k+l)x + 1\bigr) = (k-l) f_2(x) \,, \]
  so that possible reciprocal irreducible factors for $N > N_0$ must be cyclotomic
  (since they divide $f_N$ for two different~$N$) and are as follows.
  \[ \renewcommand{\arraystretch}{1.2}
     \begin{array}{@{k + l ={}}r@{{}\colon\qquad}r@{{}\mid f_N \quad}l}
       -2 &       x - 1 & \text{for all $N$} \\
        2 &       x + 1 & \text{for $N \equiv 0 \bmod 2$} \\
        1 & x^2 + x + 1 & \text{for $N \equiv 2 \bmod 3$} \\
        0 & x^2     + 1 & \text{for $N \equiv 2 \bmod 4$} \\
       -1 & x^2 - x + 1 & \text{for $N \equiv 2 \bmod 6$}
     \end{array}
  \]
  We also note that except when $k+l = -2$, we have that $f_3$ has no rational root
  and is therefore irreducible. Also, $f_4$ has no rational root unless $k+l = \pm 2$.
  It is easy to see that the only factorization of~$f_4$ as a product of two quadratics
  is (up to exchanging $k$ and~$l$)
  \[ x^4 - 3 x^3 + 3 x + 1 = (x^2 - x - 1) (x^2 - 2 x - 1) \,. \]
  Similarly, we find that the only factorization of~$f_5$ as a product of a quadratic
  and a cubic, apart from the systematically occurring factor $x^2+x+1$ when $k+l = 1$, is
  \[ x^5 - 2 x^4 + x + 1 = (x^2 - x - 1) (x^3 - x^2 - 1) \,. \]
  Except for the systematically occurring factor $x^2+1$ when $k+l = 0$, there is
  only the following factorization of~$f_6$ into a quadratic and a quartic.
  \[ x^6 - 2 x^5 + 2 x + 1 = (x^2 - x - 1) (x^4 - x^3 - x - 1) \,. \]
  There are no factorizations into two cubics. (Writing
  \[ x^6 + k x^5 + l x + 1 = (x^3 + s x^2 + t x \pm 1) (x^3 + u x^2 + v x \pm 1) \]
  and comparing coefficients gives three equations to be solved for $s,t,u,v \in \Z$.
  In both cases, the equations define an affine curve of genus~$1$. Its projective
  closure is in both cases isomorphic to the elliptic curve with label $20a4$ in the
  Cremona database, which has exactly two rational points, both of which are
  at infinity for our affine models).

  If we exclude for now the cases with $\max\{|k|, |l|\} \le 2$, then Lemma~\ref{L:linear}
  tells us that $m_0 \le 3$ in all cases of interest. So we can take $N_0 = 6$
  in Theorem~\ref{T:main}. Since we have discussed the cases $3 \le N \le 6$
  above, we see that in the cases $l = k+1$ and $l = k+2$, $f_N$ is always irreducible,
  and in the cases $-2 \le k+l \le 2$, $f_N$ factors as the cyclotomic factor given
  above times an irreducible polynomial, with the only exception of~$f_4$ when $(k,l) = \pm(-3,3)$.

  The remaining cases (with $-2 \le k < l \le 2$) can be dealt with using the algorithm
  implied by the proof of Theorem~\ref{T:main}. This finally gives the following complete
  list of exceptional factorizations (for $k < l$).
  \begin{align*}
    x^4 - 3 x^3 + 3 x + 1 &= (x^2 - x - 1) (x^2 - 2 x - 1) \\
    x^5 - 2 x^4 + x + 1 &= (x^2 - x - 1) (x^3 - x^2 - 1) \\
    x^6 - 2 x^5 + 2 x + 1 &= (x^2 + 1) (x^2 - x - 1)^2 \\
    x^7 - 2 x^6 + 2 x + 1 &= (x^3 - x - 1) (x^4 - 2 x^3 + x^2 - x - 1)
  \end{align*}

  A similar analysis for
  \[ f_N = x^N + k x^{N-1} - (l x + 1) \]
  (still with $k \neq l$) gives the following cyclotomic factors.
  \[ \renewcommand{\arraystretch}{1.2}
     \begin{array}{@{k + l ={}}r@{{}\colon\qquad}r@{{}\mid f_N \quad}l}
        2 &       x + 1 & \text{for $N \equiv 1 \bmod 2$} \\
        0 & x^2     + 1 & \text{for $N \equiv 0 \bmod 4$} \\
       -1 & x^2 - x + 1 & \text{for $N \equiv 5 \bmod 6$}
     \end{array}
  \]
  The exceptional factorizations (for $k < l$) are:
  \begin{align*}
    x^5 - x^4 - 2 x - 1 &= (x^2 - x - 1) (x^3 + x + 1) \\
    x^7 - 2 x^6 - 2 x - 1 &= (x^3 - x^2 + 1) (x^4 - x^3 - x^2 - 2 x - 1) \\
    x^8 - x^7 - x - 1 &= (x^2 + 1) (x^3 - x^2 + 1) (x^3 - x - 1)
  \end{align*}
\end{example}

\begin{example}
  Let $k, l \in \Z$ be nonzero and coprime. Then $x^N + k x + l$ is irreducible
  for all but finitely many~$N$ if and only if $k + l \neq -1$,
  $|k - l| \neq 1$ and $(k,l) \neq (1, 1)$, $(-1, 1)$, $(1, -1)$.
  We have seen in Section~\ref{S:Lj} that the polynomial is reducible
  for $N$ in certain residue classes if $(k,l) = (1,1)$, $(-1,1)$ or~$(1,-1)$.
  If $k + l = -1$, then $f_N(1) = 0$ for all~$N$. If $k - l = \pm 1$,
  then $f_N(-1) = 0$ for all even or all odd~$N$.

  So it remains to show that $f_N$ is irreducible for all large~$N$
  when $(k,l)$ is not one of the exceptional pairs. We have seen this
  for $(k,l) = (-1,-1)$ in Section~\ref{S:Lj}. In general,
  \[ r(x) = k l x^2 + (k^2 + l^2 - 1) x + k l \,, \]
  so the only possible cyclotomic divisors are $\Phi_1$, $\Phi_2$,
  $\Phi_3$, $\Phi_4$ and~$\Phi_6$. In the first case, $f_N(1) = 0$
  for infinitely many~$N$, which is equivalent to $k + l = -1$.
  In the second case, $f_N(-1) = 0$ for infinitely many~$N$,
  which is equivalent to $k - l = \pm 1$. In the last three cases,
  $r$ must be proportional to $x^2 + x + 1$, $x^2 + 1$ or $x^2 - x + 1$,
  respectively. This forces $(k,l) = (\pm 1, \pm 1)$, which are the
  cases dealt with in Section~\ref{S:Lj}. In all other cases, $f_N$
  must be irreducible for $N$~large.

\begin{figure}[htb]
  \begin{center}\Gr{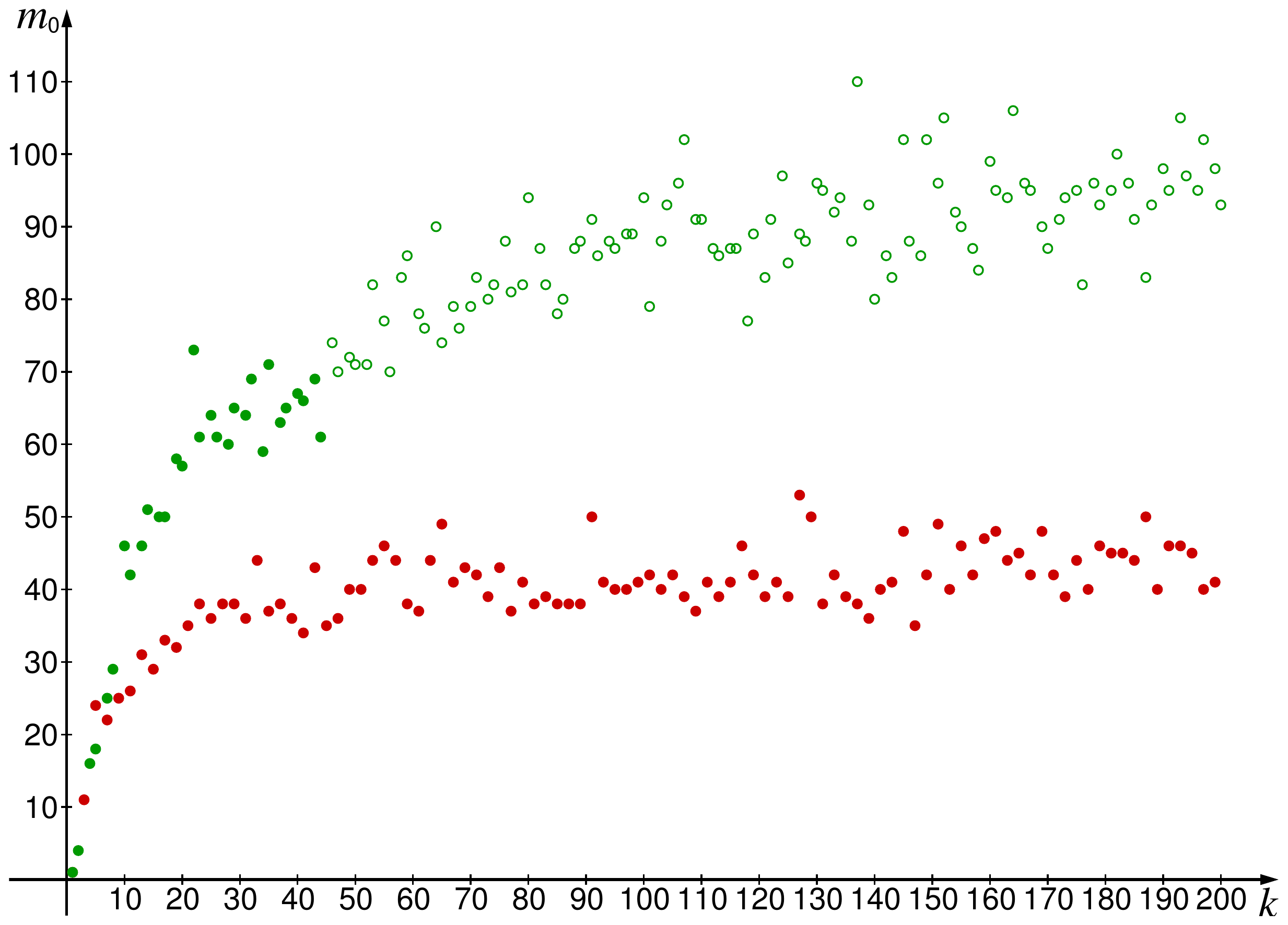}{0.75\textwidth}\end{center}
  \caption{Values of~$m_0$ for $c = 1$, $d = 2 + kx$ with $k$ odd (red)
           and for $c = 1$, $d = 3 + kx$ with $3 \nmid k$ (green), $1 \le k \le 200$.
           The values indicated by the hollow dots are likely, but not proven to be true
           (the sets~$T_m$ were cut to the $10\,000$ pairs of smallest weight).}
  \label{fig1}
\end{figure}

  We note that by Perron's criterion (see Example~\ref{Ex3}),
  the polynomials \hbox{$x^N \pm (k x + 1)$} are irreducible for all $N \ge 2$
  when $|k| \ge 3$. (The criterion is not applicable when $|l| \ge 2$,
  since the reversed polynomial is not monic.) This was part of the discussion
  in Example~\ref{Ex3} (corresponding to $(k,l) = (0,k)$ in the notation
  used there). What makes this case particularly amenable to our method
  is the uniform (and small) bound on~$m_0$.
  We remark that for $|l| \ge 2$, the behavior of~$m_0$ does not appear
  to follow a clear pattern. For example, when $l = 2$ or~$3$, $m_0$ first grows,
  but then seems to flatten out; compare Figure~\ref{fig1}.
  However, this is misleading. Note that when $|l| \ge 2$, we have that
  $f_N(l) = 0$ for $k = -l^{N-1} - 1$, which provides a
  factor $(x - l)$ of~$f_N$ not dividing~$r$, implying that $N \le 2 m_0$.
  This shows that
  \[ \limsup_{|k| \to \infty} \frac{m_0(1, l+kx)}{\log |k|} > 0 \]
  and in particular that
  \[ \limsup_{|k| \to \infty} m_0(1, l+kx) = \infty \,. \]
\end{example}

\begin{example}
  We consider trinomials $f_N = x^N + k x^{N-1} + l$ with $k, l \in \Z$, $|l| \ge 2$
  (the case $|l| = 1$ is covered by Example~\ref{Ex3}), where $N \ge 3$. By Perron's criterion,
  $f_N$ is irreducible whenever $|k| > |l| + 1$, so we can restrict to $|k| \le |l| + 1$.
  By Lemma~\ref{L:linear2}, when $l \ge 5$, we have that $m_0(1 + kx, l) = 1$ if
  either $l$ is prime and $k \neq 0$ or $|k| > \sqrt{p^2 + l^2/p^2}$, where $p$ is the smallest
  prime divisor of~$l$. This also holds for $2 \le |l| \le 4$ with the exceptions
  $m_0(1 \pm 3 x, \pm 2) = 2$ and $m_0(1 \pm 3 x, \pm 4) = 2$.
  Except in these two cases, we can therefore take $N_0 = 2$ in Theorem~\ref{T:main}.
  This is still true for $(k,l) = (\pm 3, \pm 2)$. In the other exceptional case,
  we have the factorizations
  \begin{equation} \label{E:exfact}
    x^3 \pm (3 x^2 - 4) = (x \pm 2)^2 (x \mp 1) \,.
  \end{equation}

  So for $N \ge 3$ and $(k,l) \neq \pm(3,-4)$,
  the only possible low degree factors of~$f_N$ must divide
  $r = k x^2 + (k^2 + 1 - l^2) x + k$. We can have a factor $x \pm 1$ when
  $k = -l-1$ (then $x-1$ divides~$f_N$ for all~$N$) or $k = \pm l + 1$
  (then $x+1$ divides~$f_N$ for all even~$N$ or all odd~$N$). The only other
  cyclotomic factors possible are $x^2 + 1$, $x^2 + x + 1$ and $x^2 - x + 1$;
  their occurrence would imply that $l^2 = k^2 + 1$ or $l^2 = k^2 \mp k + 1$,
  which is impossible for $l \neq \pm 1$.

  Any root of~$f_N$ other than~$\pm 1$ must be a divisor~$d$ of~$l$ with $|d| \ge 2$;
  $d$ must also be a root of~$r$. The latter implies that $d$ divides~$k$, so
  we can write $k = \kappa d$ with $\kappa \in \Z$. Then $r(d) = 0$ implies that
  $l^2 = (\kappa + 1)(\kappa d^2 + 1)$, whereas $f_N(d) = 0$ implies that
  $-l = d^N(\kappa + 1)$. Combining these relations, we get that
  $d^N l + \kappa d^2 + 1 = 0$, implying that $d^2 \mid 1$, a contradiction.

  Combining this reasoning with Perron's criterion (for $|k| > |l| + 1$),
  we obtain the following.

  \begin{proposition}
    Let $k, l \in \Z$ with $k \neq 0$, $|l| \ge 2$ and either $|l|$ prime
    or $|k| > \sqrt{p^2 + l^2/p^2}$, where $p$ is the smallest prime divisor of~$l$.
    Then for $N \ge 3$, the polynomial
    \[ f_N = x^N + k x^{N-1} + l \]
    is either irreducible or factors as $x \pm 1$ times an irreducible polynomial,
    except for the factorizations given in~\eqref{E:exfact}.
  \end{proposition}

  This improves on Theorem~1 in~\cite{Harrington2012}, where the assumption
  $2|k| \ge |l| + 2$ is made, which is stronger than our assumption when $l$ is odd.
\end{example}

\begin{example}
  Define a sequence of polynomials with coefficients in~$\{0,1\}$
  by $h_0 = 1$ and $h_{n+1} = h_n + x^k$, where $k > \deg h_n$ is minimal
  with the property that $h_n + x^k$ is reducible. Then
  \[ h_7 = x^{35} + x^{34} + x^{33} + x^{32} + x^{16} + x^{15} + x^3 + 1 \,. \]
  In~\cite{FFN} it is shown that $h_8$ does not exist. This can also be
  deduced from our main result, as follows. We consider
  \[ f_N = x^N + h_7
         = x^N + x^{35} + x^{34} + x^{33} + x^{32} + x^{16} + x^{15} + x^3 + 1
     \qquad\text{for $N \ge 36$,}
  \]
  so $c = 1$ and
  \[ d = h_7 = (x + 1) (x^{34} + x^{32} + x^{15} + x^2 - x + 1) =: a b \,, \]
  where both factors are irreducible. We have that $\|c\| + \|d\| = 9$
  and $\|a\| + \|b\| = 8$; since $b \neq -a$, our pair $(c,d)$
  is robust. Also,
  \[ r = d \tilde{d} - x^{35} = \Phi_7 h \]
  with a non-cyclotomic factor~$h$ of degree~$64$.
  Using that $d \equiv 1 \bmod \Phi_7$, it is easy to see that
  $\Phi_7$ never divides~$f_N$.
  It is also not hard to verify that $h$ never divides~$f_N$:
  we note that $h$ has a complex root $\alpha$ of absolute
  value $\approx 1.125$. Comparing the logarithms of $|d(\alpha)|$ and~$|\alpha|$
  shows that $\alpha^N + d(\alpha) = 0$ has no solution~$N \in \Z_{\ge 36}$.
  So Theorem~\ref{T:main} tells us that $f_N$ is irreducible for all
  $N > N_0$, and we can determine a suitable~$N_0$, as follows.
  We compute $m_0 = 48$ with the method sketched after Corollary~\ref{C:tree}.
  The proof of Lemma~\ref{L:trick} shows that $N_0 = \max\{4 \deg(d), 2 m_0\} = 140$
  is sufficient. We then check that $f_N$ is also irreducible for $N \le N_0$.
  (We note that the proof in~\cite{FFN} relies on similar ideas; see~\cite{Filaseta}.)
\end{example}

%==========================================================================

\end{document}